\documentclass[oneside,12pt,reqno]{amsart}

\usepackage{amssymb,amsmath}
\usepackage{amssymb,latexsym}
\usepackage{amsfonts}
\usepackage{amssymb}
\usepackage{stmaryrd}
\usepackage{longtable}
\usepackage{hyperref}
\usepackage{graphicx}

\usepackage{listings}

\usepackage{amsmath, geometry, amssymb} 
\numberwithin{equation}{section}
\newtheorem{theorem}{Theorem}[section]
\newtheorem{lemma}{Lemma}[section]
\newtheorem{corollary}{Corollary}[section]

\newtheorem{remark}{Remark}[section]

\DeclareMathOperator{\lcm}{lcm}

\def\bals#1\nals{\begin{align*}#1\end{align*}}
\def\bal#1\nal{\begin{align}#1\end{align}}

\newcommand{\sqr}[2]{{\vcenter{\vbox{\hrule height#2pt
                \hbox{\vrule width#2pt height#1pt \kern#1pt
                \vrule width#2pt}\hrule height#2pt}}}}

\newcommand{\beq}{\begin{equation}}
\newcommand{\eeq}{\end{equation}}
\newcommand{\beqar}{\begin{eqnarray}}
\newcommand{\eeqar}{\end{eqnarray}}
\def\beqars{\begin{eqnarray*}}
\def\eeqars{\end{eqnarray*}}

\newcommand{\pmd}{\hspace{-3mm} \pmod}

\newcommand{\qu}[2]{\Bigl({\frac{#1}{#2}}\Bigr) }
\newcommand{\dqu}[2]{\ds{\qu{#1}{#2}}}

\def \ds{\displaystyle}

\newcommand{\nn}{\mathbb{N}}
\newcommand{\zz}{\mathbb{Z}}
\newcommand{\qq}{\mathbb{Q}}
\newcommand{\cc}{\mathbb{C}}

\newcommand{\ff}{\mathbb{F}}
\newcommand{\hh}{\mathbb{H}}

\allowdisplaybreaks

\begin{document}

\title{Projections of modular forms on Eisenstein series and its application to Siegel's formula}%
\author{Zafer Selcuk Aygin}%
\address{Department of Mathematics and Statistics, University of Calgary, Calgary, AB T2N 1N4, Canada}%
\email{selcukaygin@gmail.com}%
\subjclass[2010]{11F11, 11F20,  11F27, 11E20, 11E25, 11F30}%
\dedicatory{Dedicated to Professor Emeritus Kenneth S.\ Williams on the occasion of his $80^{th}$ birthday.}%
\keywords{Dedekind eta function; theta functions; Eisenstein series; modular forms; cusp forms; Fourier coefficients.}

\maketitle

\begin{abstract}
Let $k \geq 2$ and $N$ be positive integers and let $\chi$ be a Dirichlet character modulo $N$. Let $f(z)$ be a modular form in $M_k(\Gamma_0(N),\chi)$. Then we have a unique decomposition $f(z)=E_f(z)+S_f(z)$, where $E_f(z) \in E_k(\Gamma_0(N),\chi)$ and $S_f(z) \in S_k(\Gamma_0(N),\chi)$. In this paper we give an explicit formula for $E_f(z)$ in terms of Eisenstein series. Then we apply our result to certain families of eta quotients and to representations of positive integers by $2k$--ary positive definite quadratic forms in order to give an alternative version of Siegel's formula for the weighted average number of representations of an integer by quadratic forms in the same genus. Our formula for the latter is in terms of generalized divisor functions and does not involve computation of local densities.
\end{abstract}
\noindent


\section{Introduction and notation} \label{sec_intro}
Let $\nn$, $\nn_0$, $\zz$, $\qq$, $\cc$ and $\hh$ denote the sets of positive integers, non-negative integers, integers, rational numbers, complex numbers and upper half plane of complex numbers, respectively. Throughout the paper $z$ denotes a complex number in $\hh$, $p$ always denotes a prime number, all divisors considered are positive divisors, $q$ stands for $e^{2 \pi i z}$ and $\chi_d(n)$ denotes the Kronecker symbol $\ds \dqu{d}{n}$, we use the subscript $K$ to avoid confusion with fractions.  Let $N \in \nn$ and $\chi$ be a Dirichlet character modulo $N$. The space of modular forms of weight $k$ for $\Gamma_0(N)$ with character $\chi$ is denoted by $M_{k}(\Gamma_0(N),\chi)$; and $E_{k}(\Gamma_0(N),\chi)$, $S_{k}(\Gamma_0(N),\chi)$ denote its Eisenstein and cusp form subspaces, respectively. Then we have
\begin{align*}
M_k(\Gamma_0(N),\chi) = E_k(\Gamma_0(N),\chi) \oplus S_k(\Gamma_0(N),\chi).
\end{align*}
That is, given $f(z) \in M_k(\Gamma_0(N),\chi)$, we can write
\begin{align}
f(z)=E_f(z) + S_f(z), \label{decomp}
\end{align}
where $E_f(z) \in E_k(\Gamma_0(N),\chi)$ and $S_f(z) \in S_k(\Gamma_0(N),\chi)$ are uniquely determined by $f$. Let $ \epsilon, \psi$ be primitive Dirichlet characters such that $\epsilon {\psi} = \chi$ (i.e., $\epsilon(n) {\psi}(n) = \chi(n)$ for all $n \in \zz$ coprime to $N$) with conductors say $L$ and $M$, respectively and suppose ${LM} \mid N $. Let $ d $ be a positive divisor of $ {N}/{LM}$ and $2\leq k \in \nn$ be such that $\chi(-1)=(-1)^k$. Let $\omega$ be the primitive Dirichlet character corresponding to $\epsilon \overline{\psi}$ and $\mathcal{M}_\omega$ be its conductor. We define the Eisenstein series associated with $\epsilon$ and ${\psi}$ by
\begin{align}
{E_{k}(\epsilon,\psi; dz)} :=& \epsilon(0)+ \left(\frac{\mathcal{M}_\omega }{M } \right)^k  \left( \frac{ W(\overline{\psi})}{W(\omega)} \right)\left( \frac{-2k}{B_{k,\overline{\omega}}} \right) \prod_{p \mid \lcm(L,M)} \frac{p^k}{p^k-\omega(p)} \\
& \qquad\qquad \times \sum_{n =1}^{\infty} \sigma_{k-1}(\epsilon, {\psi}; n) e^{2\pi i n d z} \label{eis},
\end{align}
where $\overline{\omega}$ is the complex conjugate of $\omega$,
\begin{align*}
& \sigma_{k-1}(\epsilon,\psi; n) :=  \sum_{1 \leq d\mid n}\epsilon(n/d){\psi}(d)d^{k-1} 
\end{align*}
is the generalized sum of divisors function associated with $\epsilon$ and $\psi$,
\begin{align*}
& W(\psi)  := \sum_{a=0}^{M-1} \psi(a) e^{2\pi i a/M}
\end{align*}
is the Gauss sum of $\psi$ and $B_{k,\overline{\omega}}$ is the $k$-th generalized Bernoulli number associated with $\overline{\omega}$ defined by
\bals
\sum_{k = 0}^\infty \frac{B_{k,\overline{\omega}}}{k !} t^k = \sum_{a=1}^{\mathcal{M}_{\overline{\omega}}} \frac{{\overline{\omega}}(a) t e^{at}}{e^{\mathcal{M}_{{\overline{\omega}}} t } -1},
\nals
see \cite[end of pg.\ 94]{miyake}. By \cite[Corollary 8.5.5]{cohenbook} (alternatively \cite[Theorem 4.7.1, (7.1.13) and Lemma 7.2.19]{miyake}), we have
\bals
E_{k}(\epsilon,\psi; dz ) \in M_k(\Gamma_0(N),\chi) & \mbox{ if $(k, \epsilon,\psi) \neq (2,\chi_1,\chi_1)$,}
\nals
and
\bals
E_{2}(\chi_1,\chi_1;z) - N E_{2}(\chi_1,\chi_1;Nz)  \in M_2(\Gamma_0(N),\chi_1).
\nals
\begin{remark} Let $L(\epsilon \overline{\psi}, k)$ be the Dirichlet L-function defined by
\bals
L(\epsilon \overline{\psi}, k):= \sum_{n \geq 1} \frac{\epsilon(n) \overline{\psi}(n)}{ n^{k}}.
\nals
The Eisenstein series we define in \eqref{eis} is equal to $\ds {E_k(Mdz;\epsilon,\overline{\psi})}/({2 L(\epsilon \overline{\psi}, k)})$ in the notation of Theorem 7.1.3 and Theorem 7.2.12 of \cite{miyake} and to $\ds {G_k({\psi},\epsilon)(dz)}/{L(\epsilon \overline{\psi}, k)}$ in the notation of Corollary 8.5.5 and Definition 8.5.10 of \cite{cohenbook}. Further treatment to obtain the form in \eqref{eis} is done by using the formula for $L(\epsilon \overline{\psi}, k)$ given in Theorem 3.3.4 and (3.3.14) of \cite{miyake}. 
This normalization is chosen so that the constant terms given in \eqref{7_6} and \eqref{7_4} are simpler. This in return simplifies the notation in Sections \ref{sec_orth} and \ref{sec_ft}.
\end{remark}
Letting $D(N,\cc)$ to denote the group of Dirichlet characters modulo $N$, we define
\begin{align*}
\mathcal{E}(k,N,\chi):= & \{ (\epsilon,\psi) \in D(L,\cc) \times D(M,\cc) : \mbox{$\epsilon$, $\psi$ primitive, }  \\
& \quad \epsilon(-1) \psi(-1)=(-1)^k,~ \epsilon {\psi} = \chi  \mbox{ and } LM \mid N \}.
\end{align*}
The set
\begin{align*}
& \{ E_{k}(\epsilon,\psi; dz) : (\epsilon,\psi) \in \mathcal{E}(k,N,\chi), d \mid N/LM \}
\end{align*}
constitutes a basis for $E_k(\Gamma_0(N),\chi)$ whenever $k \geq 2$ and $(k,\chi) \neq (2,\chi_1)$; the set
\begin{align*}
& \{ E_{2}(\chi_1,\chi_1;z )-d E_{2}(\chi_1,\chi_1;dz) : 1< d \mid N/LM \} \\
& \qquad \cup \{ E_{2}(Mdz; \epsilon,\psi) : (\epsilon,\psi) \in \mathcal{E}(2,N,\chi_1), (\epsilon,\psi) \neq ( \chi_1,\chi_1), d \mid N/LM \}
\end{align*}
constitutes a basis for  $E_2(\Gamma_0(N),\chi_1)$, see \cite[Theorems 8.5.17 and 8.5.22]{cohenbook}, or \cite[Proposition 5]{weis}. Then we have
\begin{align}
E_f(z) = \sum_{(\epsilon, \psi) \in \mathcal{E}(k,N,\chi) } \sum_{d \mid N/LM} a_f(\epsilon,\psi,d) E_k(\epsilon,\psi;dz), \label{RNT1}
\end{align}
for some $a_f(\epsilon,\psi,d) \in \cc$. When $S_f(z)=0$, it is easy to determine $a_f(\epsilon,\psi,d)$ by comparing the first few Fourier coefficients of $f(z)$ expanded at $i \infty$ and the first few Fourier coefficients of the right hand side of \eqref{RNT1} expanded at $i \infty$. However, if $S_f(z)\neq 0$ and an explicit basis for $S_k(\Gamma_0(N),\chi)$ is not known then this method fails. In this paper we solve this problem, in other words we obtain $a_f(\epsilon,\psi,d)$ explicitly in an accessible form for all $f \in M_k(\Gamma_0(N),\chi)$ where $k \geq 2$, see Theorem \ref{mainth}. Our treatment is general and its special cases agree with previously known formulas. Additionally, we give a new treatment of Siegel's formula for representation numbers of quadratic forms, see Theorem \ref{thetamain}.

Let $a \in \zz$ and $c \in \nn_0$ be coprime. For an $f(z) \in M_{k}(\Gamma_0(N),\chi)$ we denote the constant term of $f(z)$ in the Fourier expansion of $f(z)$ at the cusp $a/c$ by
\begin{align*}
[0]_{a/c} f = \lim_{z \rightarrow i \infty} (c z + d)^{-k} f \left( \frac{az+b}{cz+d} \right),
\end{align*}
where $b,d \in \zz$ such that $\begin{bmatrix} a & b \\ c & d \end{bmatrix} \in SL_2(\zz)$. The value of $[0]_{a/c} f $ does not depend on the choice of $b, d$. We denote the $n$th Fourier coefficient of $f(z)$ in the expansion at the cusp $i\infty$ by $[n] f$. Letting $\phi(n)$ denote the Euler totient function, we define an average associated with $\psi$ for the constant terms of Fourier series expansions of modular forms at cusps as follows:
\begin{align}
[0]_{c,{\psi}}f := \frac{1}{\phi(c)} \sum_{\substack{a =1,\\ \gcd(a,c)=1} }^c {\psi}(a)[0]_{a/c}f. \label{newcusp}
\end{align}
We note that working with this average of constant terms at cusps is a new idea which helps studying modular form spaces with nontrivial character, see Section \ref{sec_ft} for details.

Letting $v_p(n)$ to denote the highest power of $p$ dividing $n$ and $\mu(n)$ to be the M\"obius function we are ready to state the main theorem.
\begin{theorem}[Main Theorem] \label{mainth}
Let $f(z) \in M_k(\Gamma_0(N),\chi)$, where $N,k \in \nn$, $k \geq 2$, $\chi$ is a Dirichlet character modulo $N$ that satisfies $\chi(-1)=(-1)^k$. Let $E_f(z)$ be defined by \eqref{decomp}, then 
\begin{align*}
E_f(z) = \sum_{(\epsilon, \psi) \in \mathcal{E}(k,N,\chi) } \sum_{d \mid N/LM} a_f(\epsilon,\psi,d) E_k(\epsilon,\psi;dz),
\end{align*}
where 
{ \begin{align*}
a_f(\epsilon,\psi,d) & = \prod_{p \mid N} \frac{p^k}{p^k - \epsilon(p) \overline{\psi}(p)} \sum_{c \in C_{N}(\epsilon,\psi) }  \mathcal{R}_{k,\epsilon,\psi}(d,c/M) \mathcal{S}_{k,N/LM,\epsilon,\psi}(d,c/M) [0]_{c,{\psi}}f,
\end{align*}}
with
\begin{align}
C_{N}(\epsilon,\psi) & :=  \{ c_1 M :  c_1 \mid N/LM\}, \label{CMdef} \\
\mathcal{R}_{k,\epsilon,\psi}(d,c) & := \epsilon \left( \frac{-d}{\gcd(d,c)} \right) \overline{\psi} \left( \frac{c}{\gcd(d,c)} \right) \left(\frac{\gcd(d,c)}{c} \right)^k,\label{RR}
\end{align}
and
\begin{align}
\mathcal{S}_{k,N,\epsilon,\psi}(d,c)&:= \mu\left( \frac{dc}{\gcd(d,c)^2} \right) \prod_{\substack{p \mid \gcd(d,c),\\0<v_p(d)=v_p(c)<v_p(N)}} \left( \frac{p^k +\epsilon(p) \overline{\psi}(p)}{p^k} \right).\label{SS}
\end{align}
\end{theorem}

\begin{remark}
Let $c \mid N$, by Lemma \ref{reducer}, if $a/c$ and $a'/c$ are equivalent cusps of $\Gamma_0(N)$ and $(\epsilon,\psi) \in \mathcal{E}(k,N,\chi)$ with $M \mid c$ then $ \psi(a)[0]_{a/c} f =\psi(a')[0]_{a'/c} f$. 
Therefore in applications of Theorem \ref{mainth} computing $ \psi(a)[0]_{a/c} f$ at a set of inequivalent cusps will be sufficient, see \cite[Corollary 6.3.23]{cohenbook} for a description of such a set.
\end{remark}

Theorem \ref{mainth} agrees with and extends previously known formulas. For example, if we let $k \in \nn$ even, $N$ squarefree, $\chi=\chi_1$ in Theorem \ref{mainth}, we obtain \cite[Theorem 1.1]{sqfreepaper} and if we let $k \in \nn$ odd $N \in \{3,7,11,23 \}$, $\chi=\chi_{-N}$ in Theorem \ref{mainth}, we obtain \cite[(11.20)]{cooperbook}. Theorem \ref{mainth} additionally extends the latter to hold for all primes $N$ that are congruent to $3$ modulo $4$. 

Before we apply Theorem \ref{mainth} to representation numbers of quadratic forms we give a snapshot of interesting applications. Since $f(z) -E_f(z)$ is a cusp form, one can use our main theorem to produce cusp forms. At the end of Section \ref{classical} we use this idea combined with the Modularity Theorem (\cite[Theorem 8.8.1]{DiamondShurman}) and consider the elliptic curve $E_{27A}: y^2  + y=x^3 - 7$. Then we use arithmetic properties of Eisenstein series to obtain
\begin{align*}
\# E_{27A}(\ff_p) & \equiv 0  \pmd{9}  \mbox{ if $p \equiv 1 \pmd{3}$,} \\
\# E_{27A}(\ff_p) & = p+1  \mbox{ if $p \equiv 2 \pmd{3}$,}
\end{align*}
where
\begin{align*}
E_{27A}(\ff_p):= \{ \infty \} \cup \{ (x,y) \in \ff_p \times \ff_p : y^2  + y=x^3 - 7 \}, 
\end{align*}
with $\ff_p$ denoting the finite field of $p$ elements, see Corollary \ref{cor2}.

The Fourier coefficients of special functions (expanded at $i \infty$) have been of huge interest. A very well studied special function is the Dedekind eta function which is defined by
\begin{align*}
\eta(z):=e^{\pi i z/12} \prod_{n \geq 1} (1-e^{2 \pi i n z})=q^{1/24} \prod_{n \geq 1} (1-q^{n}).
\end{align*}
Quotients of Dedekind eta functions are often referred to as eta quotients. Nathan Fine in his book \cite{fine} has given several formulas for Fourier coefficients of eta quotients (expanded at $i \infty$). In his work when the weight of the eta quotient is integer, the formulas are linear combinations of Eisenstein series defined above. For instance he shows that
{\begin{align*}
F(z):=\frac{\eta(2z)\eta(3z)\eta(8z)\eta(12z)}{\eta(z)\eta(24z)}=1+ \sum_{n \geq 1} \sigma_{0}(\chi_1,\chi_{-24};n) q^n,
\end{align*}}%
see \cite[(32.5)]{fine}, and acknowledges this equation as being {\it very beautiful}. We consider the $(2k+1)$th power of $F(z)$, that is we consider
\begin{align*}
F^{2k+1}(z)=\frac{\eta^{2k+1}(2z)\eta^{2k+1}(3z)\eta^{2k+1}(8z)\eta^{2k+1}(12z)}{\eta^{2k+1}(z)\eta^{2k+1}(24z)}.
\end{align*}
Using our main theorem (Theorem \ref{mainth}) we obtain the following analogous formula for $F^{2k+1}(z)$ when $k > 0$:
\begin{align*}
E_{F^{2k+1}}(z) & =1 - \frac{2k+1}{B_{{2k+1},\chi_{-24}}} \sum_{n \geq 1} \left( \sigma_{{2k}}(\chi_1,\chi_{-24};n) + (-24)^{k} \sigma_{{2k}}(\chi_{-24},\chi_1;n) \right) q^n,
\end{align*}
see Corollary \ref{cor_1}. Using Theorem \ref{mainth} one can obtain formulas in this fashion for all holomorphic eta quotients of integral weight $k \geq 2$.

Let $\mathcal{F}(x_1,\ldots,x_{2k})$ be a positive definite quadratic form with integer coefficients and $B(\mathcal{F})$ be the matrix associated with $\mathcal{F}$ whose entries are given by
\begin{align*}
B(\mathcal{F})_{i,j}=\left( \frac{\partial^2 \mathcal{F}}{\partial x_i \partial x_j}   \right).
\end{align*} 
Then the generating function of the number of representations of a positive integer by the quadratic form $\mathcal{F}$ is
\begin{align*}
\theta_\mathcal{F}(z) = \sum_{x \in \zz^{2k}}e^{ 2 \pi i z \mathcal{F}(x) } = \sum_{x \in \zz^{2k}}e^{ 2 \pi i z x B(\mathcal{F}) x^{T}/2 }.
\end{align*} 
In \cite{siegel} Siegel gave a formula for the weighted average for representation numbers of positive definite quadratic forms in the same genus. Siegel's formula is in terms of local densities, for other treatments of Siegel's formula see \cite{yang} and \cite[Chapter 3]{opitz}. In the realm of modular forms, Siegel's formula corresponds to the Eisenstein part of $\theta_\mathcal{F}(z)$, see \cite{arenas}, \cite{surveyRSP}, \cite[Remark on pg 110]{serrebook} and \cite{siegel}. For clarity we note that if $\mathcal{F}_1$ and $\mathcal{F}_2$ are in the same genus then  $E_{\theta_{\mathcal{F}_1}}(z) = E_{\theta_{\mathcal{F}_2}}(z)$. 
Below we use Theorem \ref{mainth} to give an explicit formula for $E_{\theta_{\mathcal{F}}}(z)$, where $\mathcal{F}$ is a $2k$--ary positive definite quadratic form with integer coefficients. In Section \ref{quad} we give several applications of our formula including a comparison of our output for the form $\sum_{j=1}^{2k} x_j^2$ with that of Arenas \cite[Proposition 1]{arenas}, which uses Siegel's formula. 

By \cite[Corollary 4.9.5]{miyake}, we have
\begin{align}
\theta_\mathcal{F}(z) \in M_{k}(\Gamma_0(N),\chi), \label{4_1}
\end{align}
where $\ds \chi=\dqu{(-1)^k \det(B(\mathcal{F}))}{*}$ and $N$ is the smallest positive integer such that the matrix $NB(\mathcal{F})^{-1}$ has even diagonal entries. By \cite[(10.2)]{wangpei} we have
{\begin{align}
[0]_{a/c} \theta_\mathcal{F}(z)= \left( \frac{-i}{c} \right)^k \frac{1}{\sqrt{\det(B(\mathcal{F}))}} \sum_{\substack{x \in \zz^{2k},\\ x \pmd{c}}} e^{ 2 \pi i (\mathcal{F}(x) a/c)}. \label{4_2}
\end{align}}%
Putting \eqref{eis}, \eqref{4_1} and \eqref{4_2} in Theorem \ref{mainth}, we obtain the following assertion concerning the representation numbers of $2k$--ary quadratic forms.
\begin{theorem} \label{thetamain} Let $\mathcal{F}(x_1,\ldots,x_{2k})$ be a positive definite quadratic form with $k \geq 2$; let $\chi$ and $N$ be as above and $\omega$ be as in \eqref{eis}. Then 
\begin{align}
[n] E_{\theta_\mathcal{F}}(z) & = \sum_{(\epsilon, \psi) \in \mathcal{E}(k,N,\chi) } \left(\frac{\mathcal{M}_\omega }{M } \right)^k  \left( \frac{ W(\overline{\psi})}{W(\omega)} \right)\left( \frac{-2k}{B_{k,\overline{\omega}}} \right) \prod_{p \mid \lcm(L,M)} \frac{p^k}{p^k-\omega(p)}  \label{4_3}\\
& \qquad \qquad \qquad \times \sum_{d \mid N/LM} a_{\theta_\mathcal{F}}(\epsilon,\psi,d)  \sigma_{k-1}(\epsilon, {\psi}; n/d) , \nonumber
\end{align}
where
\begin{align*}
& a_{\theta_\mathcal{F}}(\epsilon,\psi,d) =  \frac{(-i)^k}{\sqrt{\det(B(\mathcal{F}))}} \prod_{p \mid N} \frac{p^k}{p^k - \epsilon(p) \overline{\psi}(p)}  \\
& \qquad \times \sum_{c \in C_{N}(\epsilon,\psi) } \frac{ \mathcal{R}_{k,\epsilon,\psi}(d,c/M) \mathcal{S}_{k,N/LM,\epsilon,\psi}(d,c/M)}{c^k \phi(c)} \sum_{ \substack{a=1,\\ \gcd(a,c)=1}}^c  {\psi}(a) \sum_{\substack{ x \in \zz^{2k},\\ x \pmd{c}}} e^{ 2 \pi i (\mathcal{F}(x) a/c)}.
\end{align*}
\end{theorem}%

The organization of the rest of the paper is as follows. In Section \ref{quad}, we apply  Theorem \ref{thetamain} to the representation numbers of diagonal quadratic forms and certain non-diagonal level $2$ quadratic forms. A special case of the latter leads to an equation for Ramanujan's tau function. In Section \ref{classical}, we apply our Main Theorem to certain families of eta quotients, these applications give extensions of some well known formulas to higher weight eta quotients. In Sections \ref{sec_orth}--\ref{sec_proof} we prove the main theorem. 

\section{Applications to representation numbers of certain quadratic forms} \label{quad}

To apply \eqref{4_3} to specific quadratic forms we need to compute the quadratic Gauss sum. If $\mathcal{F}$ is a diagonal form, say $\mathcal{F}=\sum_{j=1}^{2k} \alpha_j x_j^2$, then we have
{ \begin{align}
\sum_{\substack{ x \in \zz^{2k},\\ x \pmd{c}}} e^{2 \pi i \mathcal{F}(x)a/c} =  \prod_{j=1}^{2k} \gcd(\alpha_j a,c) g\left( \frac{\alpha_j a}{\gcd(\alpha_j a,c)},\frac{c}{\gcd(\alpha_j a,c)} \right), \label{4_4} 
\end{align}}%
where, if $\gcd(\alpha, \beta)=1$,
\begin{align*}
g(\alpha, \beta)= \begin{cases}  0 & \mbox{ if $\beta \equiv 2 \pmd{4}$,} \\
 \dqu{\alpha}{\beta} \sqrt{\beta}  & \mbox{ if $\beta \equiv 1 \pmd{4}$,}\\
i  \dqu{\alpha}{\beta} \sqrt{\beta} & \mbox{ if $\beta \equiv 3 \pmd{4}$,}\\
 (1+i)  \dqu{\beta}{\alpha} \sqrt{\beta} & \mbox{ if $\beta \equiv 0 \pmd{4}$ and $\alpha \equiv 1 \pmd{4}$,}\\
(1- i) \dqu{\beta}{\alpha} \sqrt{\beta} & \mbox{ if $\beta \equiv 0 \pmd{4}$ and $\alpha \equiv 3 \pmd{4}$,}
 \end{cases}
 \end{align*}
see \cite[Theorems 1.5.2 and 1.5.4]{kennethbook}. Next we apply this result to the form $\mathcal{F}=\sum_{j=1}^{2k} x_j^2$, that is, $\alpha_j=1$ for all $1 \leq j \leq 2k$. Then we have
\begin{align*}
\sum_{\substack{ x \in \zz^{2k},\\ x \pmd{c}}} e^{2 \pi i \mathcal{F}(x)/c}=\prod_{i=1}^{2k}  g\left( 1,c  \right) = \begin{cases}
1 & \mbox{if $c=1$,}\\
0 & \mbox{if $c=2$,}\\
(8i)^{k} & \mbox{if $c=4$.} \end{cases}
\end{align*}
Thus by Theorem \ref{thetamain} when $k$ is even we have 
{ \begin{align}
E_{\theta_{\mathcal{F}}}(z) & = 1-  \frac{2k}{(2^k-1)B_{k,\chi_1}} \sum_{n \geq 1} \left( (-i)^k \sigma(\chi_1,\chi_1;n) - (i^k+1) \sigma(\chi_1,\chi_1;n/2) \right. \nonumber \\
& \qquad \qquad\qquad\qquad \qquad\qquad \left. + 2^k \sigma(\chi_1,\chi_1;n/4)  \right) q^n \label{eqm:1}
\end{align}}%
and when $k$ is odd we have
\begin{align}
E_{\theta_{\mathcal{F}}}(z)= 1- \frac{2k}{B_{k,\chi_{-4}}} \sum_{n \geq 1} \left( \sigma(\chi_1,\chi_{-4};n) + (2i)^{k-1} \sigma(\chi_{-4},\chi_1;n)  \right) q^n. \label{eqm:2}
\end{align}
\eqref{eqm:1} and \eqref{eqm:2} agrees with Ramanujan's statements \cite[(131)--(134)]{ramanujan}, which was first proven by Mordell in \cite{mordell}. In \cite[Proposition 1]{arenas} Arenas uses Siegel's formula  to compute $E_{\theta_{\mathcal{F}}}(z)$ and obtains \eqref{eqm:1} and \eqref{eqm:2} in the same form. Now we turn our attention to another diagonal form. Let 
\begin{align*}
\mathcal{F}(a,b;p)=\sum_{i=1}^{a} x_i^2 + \sum_{i=1}^{b} p y_i^2.
\end{align*}
In \cite{cooperrmf} Cooper, Kane and Ye found formulas for the representation numbers of $\mathcal{F}(k,k;p)$, where $p=3,7,11$ or $23$. Their result relies on the existence of a Hauptmodul in the levels considered. Inspired by their results, in \cite{rmfpaper}, we derived formulas for the representation numbers of $\mathcal{F}(2a,2b;p)$ where $a,b \in \nn_0$ and $p$ is an odd prime. These results are considered as analogues of the Ramanujan-Mordell formula and specialized version of Theorem \ref{thetamain} agrees with these results. Below we give formulas in all the remaining cases, that is, we find formulas for representation numbers of $\mathcal{F}(a,b;p)$ where $a,b \equiv 1 \pmod{2}$ and $p$ an odd prime. 

\begin{corollary} \label{cor_4_1}
Let $a,b \geq 1$ be odd integers such that $a+b \geq 4$. Set $k=(a+b)/2$ and $\mathbf{p}=\chi_{-4}(p) p$. Then for any odd prime $p$, whenever $(-1)^k = \chi_{-4}(p)$ we have
{\begin{align*}
& E_{\theta_{\mathcal{F}(a,b;p)}}(z) =  \\
& 1 +  \sum_{n =1}^{\infty} \frac{2k \left( a_1 \sigma_{k-1}(\chi_1,\chi_{\mathbf{p}}; n) + a_2 \sigma_{k-1}(\chi_1,\chi_{\mathbf{p}}; n/2) + a_3 2^{k} \sigma_{k-1}(\chi_1,\chi_{\mathbf{p}}; n/4) \right)}{({2^k - \chi_{\mathbf{p}}(2)}) B_{k,\chi_{\mathbf{p}}} } q^{ n }\\
& +   \sum_{n =1}^{\infty} \frac{p^{(a-1)/2} 2k\left( a_4 \sigma_{k-1}(\chi_{\mathbf{p}},\chi_{1}; n) +  a_5 \sigma_{k-1}(\chi_{\mathbf{p}},\chi_{1}; n/2) +  a_6 2^k \sigma_{k-1}(\chi_{\mathbf{p}},\chi_{1}; n/4)  \right)}{(2^k - \chi_{\mathbf{p}}(2)) B_{k,\chi_{\mathbf{p}}}} q^{n},
\end{align*} }%
and whenever $(-1)^k = -\chi_{-4}(p)$ we have
\begin{align*}
E_{\theta_{\mathcal{F}(a,b;p)}}(z)  &=  1- \sum_{n =1}^{\infty} \frac{2k\left( b_1 \sigma_{k-1}(\chi_{1},\chi_{-4\mathbf{p}}; n) +   b_2 2^k \sigma_{k-1}(\chi_{-4},\chi_{\mathbf{p}}; n) \right)}{B_{k,\chi_{-4\mathbf{p}}}} q^{n}  \\
&  -   \sum_{n =1}^{\infty} \frac{p^{(a-1)/2} 2k\left( b_3 \sigma_{k-1}(\chi_{\mathbf{p}},\chi_{-4}; n) + b_4 2^k  \sigma_{k-1}(\chi_{-4\mathbf{p}},\chi_{1}; n)  \right)}{B_{k,\chi_{-4\mathbf{p}}}} q^{n}
\end{align*}
where 
\begin{align*}
a_1 & = \begin{cases}
(-1)^{k/2} & \mbox{ if $ p \equiv 1 \pmd{4} $,}\\
(-1)^{(k+a+2)/2} & \mbox{ if $ p \equiv 3 \pmd{4} $,}
\end{cases}\\
a_2 & = \begin{cases}
(-1)^{k/2+1}-\chi_{\mathbf{p}}(2)  & \mbox{ if $ p \equiv 1 \pmd{4} $,}\\
(-1)^{(k+a)/2} - \chi_{\mathbf{p}}(2)& \mbox{ if $ p \equiv 3 \pmd{4} $,}
\end{cases}\\
a_3 & = 1,\\
a_4 & = \begin{cases}
(-1)^{k/2}   & \mbox{ if $ p \equiv 1 \pmd{4} $,}\\
(-1)^{(k-1)/2} & \mbox{ if $ p \equiv 3 \pmd{4} $,}
\end{cases}\\
a_5 & = \begin{cases}
(-1)^{k/2+1} \chi_{\mathbf{p}}(2) - 1   & \mbox{ if $ p \equiv 1 \pmd{4} $,}\\
(-1)^{(b+1)/2}+(-1)^{(k+1)/2}\chi_{\mathbf{p}}(2) & \mbox{ if $ p \equiv 3 \pmd{4} $,}
\end{cases}\\
a_6 & = \begin{cases}
1 & \mbox{ if $ p \equiv 1 \pmd{4} $,}\\
(-1)^{(b-1)/2} & \mbox{ if $ p \equiv 3 \pmd{4} $,}
\end{cases}\\
b_1 & = 1,\\
b_2 & = \begin{cases}
{(-1)^{(k-1)/2}}/{2} & \mbox{ if $ p \equiv 1 \pmd{4} $,}\\
{(-1)^{(k+a-1)/2}}/{2} & \mbox{ if $ p \equiv 3 \pmd{4} $,}
\end{cases}\\
b_3 & = \begin{cases}
1 & \mbox{ if $ p \equiv 1 \pmd{4} $,}\\
(-1)^{(b+1)/2} & \mbox{ if $ p \equiv 3 \pmd{4} $,}
\end{cases}\\
b_4 & = \begin{cases}
(-1)^{(k-1)/2}/2 & \mbox{ if $ p \equiv 1 \pmd{4} $,}\\
(-1)^{k/2}/2 & \mbox{ if $ p \equiv 3 \pmd{4} $.}
\end{cases}
\end{align*}

\end{corollary}%

Using \eqref{4_4} and Theorem \ref{thetamain} one can obtain results similar to Corollary \ref{cor_4_1} for any diagonal form. Next we consider the non-diagonal form
\begin{align*}
\mathcal{F}_k= \sum_{m=1}^{k} \left( \sum_{1 \leq i \leq j \leq 4} x_{i,m} x_{j,m} \right) - x_{1,m} x_{2,m}.
\end{align*}
We obtain
\begin{align*}
& N=2, \qquad \det(B(\mathcal{F}_k))=2^{2k} \mbox{ and } \sum_{\substack{ x \in \zz^{2k},\\ x \pmd{c}}} e^{ 2 \pi i \mathcal{F}_k(x)/c}=\begin{cases}
1 & \mbox{ if $c=1$,}\\
(-8)^k & \mbox{ if $c=2$}. \end{cases}
\end{align*}
Thus $\theta_{\mathcal{F}_k} \in M_{2k}(\Gamma_0(2),\chi_1)$, hence by Theorem \ref{thetamain} we have
{ \begin{align}
[n]E_{\theta_{\mathcal{F}_k}}(z) = \frac{-4k}{((-2)^k+1)B_{2k,\chi_1}} \left( \sigma_{2k-1}(\chi_1,\chi_1;n) + (-2)^k \sigma_{2k-1}(\chi_1,\chi_1;n/2)  \right). \label{6_1}
\end{align}}%
When $k=6$ we compute the first few coefficients of the cusp part of $\theta_{\mathcal{F}_6}$:
{\begin{align*}
\theta_{\mathcal{F}_6}(z) - E_{\theta_{\mathcal{F}_6}}(z)= \frac{2^6 3^4 19}{691}q+ \frac{2^6 3^4 19 }{691} (2^6 -24) q^2+ \frac{2^6 3^419 }{691} 252 q^3 +O(q^4) \in S_{12}(\Gamma_0(2),\chi_1).
\end{align*}}
The Fourier coefficients of $\eta^{24}(z)$ are called the Ramanujan's $\tau$ function and first few terms are given as follows
\begin{align}
\eta^{24}(z)=\sum_{n \geq 1} \tau(n) q^n = q-24 q^2+252 q^3 +O(q^4). \label{RINT6_2}
\end{align}
It is well known that $\eta^{24}(z)$ and $\eta^{24}(2z) \in S_{12}(\Gamma_0(2),\chi_1)$ thus by Sturm Theorem \cite[Corollary 5.6.14]{cohenbook} we obtain
\bal
\theta_{\mathcal{F}_6}(z) - E_{\theta_{\mathcal{F}_6}}(z)=  \frac{2^6 3^4 19}{691}(\eta^{24}(z) + 2^6 \eta^{24}(2z)). \label{rama1}
\nal
If we compare $n$th coefficient of both sides of \eqref{rama1} we get
\begin{align*}
 [n]\theta_{\mathcal{F}_6}(z)-\frac{2^4 3^2 7}{691}  (\sigma_{11}(\chi_1,\chi_1;n) + 2^6 \sigma_{11}(\chi_1,\chi_1;n/2)) =  \frac{2^6 3^4 19}{691}  (\tau(n) + 2^6 \tau(n/2)).
\end{align*}
Since 
\begin{align*}
& [n]\theta_{\mathcal{F}_6}(z) \in \nn_0 \mbox{ for all $n \in \nn_0$ and } 2^4 3^2 7 \equiv -2^6 3^4 19 \pmd{691}
\end{align*}
it is not hard to deduce the well known congruence relation
\begin{align*}
\tau(n) \equiv \sigma_{11}(\chi_1,\chi_1;n) \pmd{691}.
\end{align*}

\section{Applications to eta quotients} \label{classical}

In this section we give further applications of Theorem \ref{mainth}. Recall that the Dedekind eta function is defined by
\begin{align*}
\eta(z)=e^{\pi i z/12} \prod_{n \geq 1} (1-e^{2 \pi i n z}).
\end{align*}
Let $k \in \nn$. We define
{\begin{align}
f_{k}(z)& :=\frac{\eta^{2k+1}(2z)\eta^{2k+1}(3z)\eta^{2k+1}(8z)\eta^{2k+1}(12z)}{\eta^{2k+1}(z)\eta^{2k+1}(24z)}, \label{eqr_11} \\
g_{k}(z) & :=\frac{\eta^{6{k}-5}(3z)\eta^{6{k}-4}(4z)}{\eta^{{2k}-3}(z)\eta^{{2k}-2}(2z)\eta^{{2k}-4}(6z)\eta^{{2k}}(12z)} \label{eqr_2}, \\ 
h_{k}(z) & :=\frac{\eta^{6{k}-4}(9z) \eta^{3}(27z)}{\eta^{{2k}-1}(3z)}. \label{eqr_3}
\end{align}}%
In Corollary \ref{cor_1} below we obtain formulas concerning $f_k(z), g_k(z)$ and $h_k(z)$. Similar formulas can be obtained via Theorem \ref{mainth} for all integer weight holomorphic eta quotients.

\begin{corollary} \label{cor_1} Let $k \geq 1$ and let $f_k(z), g_k(z)$ and $h_k(z)$ be defined by \eqref{eqr_11}, \eqref{eqr_2}, and \eqref{eqr_3}, respectively. Then we have
{\begin{align*}
E_{f_{k}}(z) & =1 - \frac{4k+2}{B_{{2k+1},\chi_{-24}}} \sum_{n \geq 1} \left( \sigma_{{2k}}(\chi_1,\chi_{-24};n) + (-24)^{{k}} \sigma_{{2k}}(\chi_{-24},\chi_1;n) \right) q^n, \\
E_{g_{k}}(z) & = 1- \frac{4{k}}{B_{{2k},\chi_{12}}} \sum_{n \geq 1}  \sigma_{{2k}-1}(\chi_{1},\chi_{12};n) q^n, 
\end{align*}
and
\begin{align}
E_{h_{k}}(z) & = \frac{-4{k}}{B_{k,{\chi_{1}}}} \sum_{n \geq 1} \left(  \sum_{d \mid 9} a_d  \sigma_{{2k}-1}(\chi_{1}, \chi_{1}; n/d) +    b_1\sigma_{{2k}-1}(\chi_{-3}, \chi_{-3}; n) \right) q^n, \label{eqr_4}
\end{align}}%
where
{\begin{align*}
a_1 & =  \frac{(-1)^{{k}} - \cos \left( {({k}+4) \pi}/{3} \right)}{3^{3{k}+1} (3^{2k} - 1)},   \\
a_3  & = \frac{(-1)^{{k_2}/2+1} +(3^{2k} +1) \cos \left( {({k}+4) \pi}/{3} \right)}{3^{3{k}+1} (3^{2k} - 1)},  \\
a_9  & =  \frac{- \cos \left( {({k}+4) \pi}/{3} \right)}{3^{{k}+1} (3^{2k} - 1)},   \\
b_1  &  =  \frac{ \sqrt{3}  \sin \left( {({k}+4) \pi}/{3} \right)}{3^{3{k}+1} (3^{2k}-1)}.
\end{align*}}%
\end{corollary}
\begin{proof}
We use \cite[Proposition 5.9.2]{cohenbook} to determine
\bals
f_k(z) & \in M_{2k+1}(\Gamma_0(24),\chi_{-24}),\\
g_k(z) & \in M_{2k}(\Gamma_0(12),\chi_{12}),\\
h_k(z) & \in M_{2k}(\Gamma_0(27),\chi_{1}).
\nals
We evaluate the constant terms of $f_k(z),g_k(z),h_k(z)$ at the relevant cusps using \cite[Proposition 2.1]{Kohler}. We do this with the help of some SAGE functions we have written, the code is provided in the Appendix \ref{appa}. From these we compute
\bals
& [0]_{1/1} f_k =- \frac{i^{2k+1} \sqrt{6}}{3^{k+1} 2^{3k+2}},\\
& [0]_{a/c}f_k =0, \mbox{ for $a/c= 1/2,1/3,1/4,1/6,1/8,1/12$}, \\
& [0]_{1/24}f_k =1,
\nals
see Appendix \ref{appa} for details. We determine the set of tuples of characters as
\bals
\mathcal{E}(2k+1,24,\chi_{-24}) = \{ (\chi_{1},\chi_{-24}), (\chi_{-4},\chi_{6}), (\chi_{6},\chi_{-4}),   (\chi_{-24},\chi_{1})\}.
\nals
Thus we have
\bals
E_{f_{k}}(z) = \sum_{(\epsilon, \psi) \in \mathcal{E}(2k+1,24,\chi_{-24}) } a_{f_k}(\epsilon,\psi,1) E_{2k+1}(\epsilon,\psi;z).
\nals
Now we compute
\bal
a_{f_k}(\chi_{1},\chi_{-24},1) & = \left( \prod_{p \mid 24} \frac{p^k}{p^k -\chi_{1}(p)  \chi_{-24}(p)}  \right)\mathcal{R}_{k,\chi_{1},\chi_{-24}}(1,1) \mathcal{S}_{k,1,\chi_{1},\chi_{-24}}(1,1) [0]_{24,{\chi_{-24}}}f_k \nonumber \\
 & = \mathcal{R}_{k,\chi_{1},\chi_{-24}}(1,1) \mathcal{S}_{k,1,\chi_{1},\chi_{-24}}(1,1) [0]_{24,{\chi_{-24}}}f_k \nonumber \\
 & = [0]_{24,{\chi_{-24}}}f_k. \label{eqrama1}
\nal
We further have
\bal
[0]_{24,{\chi_{-24}}}f_k = \frac{1}{\phi(24)} \sum_{\substack{a=1,\\ \gcd(a,24)=1}}^{24} \chi_{-24}(a) [0]_{a/24} f_k = \chi_{-24}(1) [0]_{1/24} f_k = 1. \label{eqrama2}
\nal
Combining \eqref{eqrama1} and \eqref{eqrama2} we have $a_{f_k}(\chi_{1},\chi_{-24},1)=1$.

The rest of the coefficients are obtained similarly.
\end{proof}
Now we turn our attention to special cases of these formulas. The dimension of $S_{2}(\Gamma_0(12),\chi_{12})$ is $0$, so we obtain an exact formula for $g_1$, i.e. we have $g_{1}(z)=E_{g_{1}}(z)$. When $k=1$, \eqref{eqr_4} specializes to
\begin{align*}
 E_{h_1}(z) & = \sum_{n \geq 1} \left(  \frac{1}{18}  \sigma_{1}(\chi_{1}, \chi_{1}; n)  -  \frac{2 }{ 9 }  \sigma_{1}(\chi_{1}, \chi_{1}; n/3) \right. \\
 & \qquad \qquad  \left. +\frac{1}{6} \sigma_{1}(\chi_{1}, \chi_{1}; n/9) +    \frac{1 }{18 } \sigma_{1}(\chi_{-3}, \chi_{-3}; n) \right) q^n.
\end{align*}
Clearly $h_1(z)-E_{h_1}(z)$ is a cusp form, and if we normalize $h_1(z)-E_{h_1}(z)$ so that the coefficient of $q$ is $1$, we obtain the newform $\mathcal{N}_{27}(z)$ in $S_{2}(\Gamma_0(27),\chi_1)$, that is, we have
{ \begin{align*}
\mathcal{N}_{27}(z) & =-9 \frac{\eta^{2}(9z) \eta^{3}(27z)}{\eta(3z)} + \sum_{n \geq 1} \left(  \frac{1}{2}  \sigma_{1}(\chi_{1}, \chi_{1}; n)  -2 \sigma_{1}(\chi_{1}, \chi_{1}; n/3) \right. \\
 &\qquad \qquad \qquad \qquad \qquad \qquad \qquad  \left.+ \frac{3}{2} \sigma_{1}(\chi_{1}, \chi_{1}; n/9)  +    \frac{1 }{2 } \sigma_{1}(\chi_{-3}, \chi_{-3}; n) \right) q^n.
\end{align*}}%
By \cite[Table 1]{cremona} this newform is associated to the elliptic curve
\begin{align*}
E_{27A}: y^2  + y=x^3 - 7.
\end{align*}
Recall that in Section \ref{sec_intro} we defined
\begin{align*}
E_{27A}(\ff_p)= \{ \infty \} \cup \{ (x,y) \in \ff_p \times \ff_p : y^2  + y=x^3 - 7 \}, 
\end{align*}
where $\ff_p$ is the finite field of $p$ elements. Then by the Modularity Theorem, see \cite[Theorem 8.8.1]{DiamondShurman}, we have
\begin{align*}
\# E_{27A}(\ff_p) =(p+1) - [p] \mathcal{N}_{27}(z)  \mbox{ for all $p \neq 3$}.
\end{align*}
Thus for all $p \neq 3$ we have
\begin{align*}
\# E_{27A}(\ff_p)  = 9 [p]\frac{\eta^{2}(9z) \eta^{3}(27z)}{\eta(3z)} + (p+1) \left( \frac{ 1- \chi_{-3}(p) }{2}  \right).
\end{align*}
Since $\ds [p]\frac{\eta^{2}(9z) \eta^{3}(27z)}{\eta(3z)} \in \zz$ for all $p \in \nn$ and $\ds [p]\frac{\eta^{2}(9z) \eta^{3}(27z)}{\eta(3z)} =0$ when $p \equiv 2 \pmod{3}$, we obtain the following statement.
\begin{corollary} \label{cor2} We have
\begin{align*}
& \# E_{27A}(\ff_p) \equiv 0 \pmd{9}  \mbox{ if $p \equiv 1 \pmd{3}$,}\\ 
& \# E_{27A}(\ff_p) = p+1  \mbox{ if $p \equiv 2 \pmd{3}$.}
\end{align*}
\end{corollary}

\section{Orthogonal relations} \label{sec_orth}

In this section we prove some orthogonal relations involving the functions $\mathcal{R}_{k,\epsilon,\psi}(d,c) $ and $ \mathcal{S}_{k,N,\epsilon,\psi}(d,c)$ defined in \eqref{RR} and \eqref{SS}, respectively. These orthogonal relations concern the constant terms of the Eisenstein series and give the means to determine $a_f(\epsilon,\psi,d)$ of Theorem \ref{mainth}. 
Throughout the section we assume $k,N \in \nn$ and $\epsilon, \psi$ are primitive Dirichlet characters with conductors $L,M$, respectively, such that $LM \mid N$.

\begin{lemma} \label{lem:3_1} Let $p \mid N$ be a prime and let $t \mid N/p^v$, where $v=v_p(N)$, then for $0\leq i \leq v$ we have
\begin{align*}
\mathcal{S}_{k,N,\epsilon,\psi}(t\cdot p^i,d)=\mathcal{S}_{k,p^v,\epsilon,\psi}(p^i,p^{v_p(d)}) \mathcal{S}_{k,N/p^v,\epsilon,\psi}(t,d/p^{v_p(d)}).
\end{align*}
\end{lemma}

\begin{proof} Since $t \mid N/p^v$ we have $\gcd(t,p)=1$. Using the multiplicative properties of the M\"obius function we obtain
{ \begin{align*}
& \mathcal{S}_{k,N,\epsilon,\psi}(t\cdot p^i,d)= \mu\left( \frac{t\cdot p^i \cdot d}{\gcd(t\cdot p^i,d)^2} \right) \prod_{\substack{p_2 \mid \gcd(t\cdot p^i,d),\\0<v_{p_2}(t\cdot p^i)=v_{p_2}(d)<v_{p_2}(N)}} \left( \frac{p_2^k +\epsilon(p_2) \overline{\psi}(p_2)}{p_2^k} \right)\\
&= \mu\left( \frac{ p^i \cdot p^{v_p(d)}}{\gcd( p^i,p^{v_p(d)})^2} \right)  \prod_{\substack{p_2 \mid \gcd( p^i,p^{v_p(d)}),\\0<v_{p_2}( p^i)=v_{p_2}(p^{v_p(d)})<v_{p_2}(N)}} \left( \frac{p_2^k +\epsilon(p_2) \overline{\psi}(p_2)}{p_2^k} \right)\\
& \qquad \times \mu\left( \frac{t\cdot d/p^{v_p(d)} }{\gcd(t ,d/p^{v_p(d)})^2} \right) \prod_{\substack{p_2 \mid \gcd(t ,d/p^{v_p(d)}),\\0<v_{p_2}(t )=v_{p_2}(d/p^{v_p(d)})<v_{p_2}(N/p^{v_p(d)})}} \left( \frac{p_2^k +\epsilon(p_2) \overline{\psi}(p_2)}{p_2^k} \right) \\
& =\mathcal{S}_{k,p^v,\epsilon,\psi}(p^i,p^{v_p(d)}) \mathcal{S}_{k,N/p^v,\epsilon,\psi}(t,d/p^{v_p(d)}).
\end{align*}}%

\end{proof}

\begin{lemma} \label{lem:3_2} Let $\gcd(t,p^i)=1$, then we have
\begin{align*}
\mathcal{R}_{k,\epsilon,\psi}(c,t \cdot p^i)& = \epsilon(-1) \mathcal{R}_{k,\epsilon,\psi}(p^{v_p(c)}, p^i) \mathcal{R}_{k,\epsilon,\psi}(c/p^{v_p(c)},t),\\
\mathcal{R}_{k,\epsilon,\psi}(t \cdot p^i,d) &=\epsilon(-1) \mathcal{R}_{k,\epsilon,\psi}( p^i, p^{v_p(d)}) \mathcal{R}_{k,\epsilon,\psi}(t, d/p^{v_p(d)}).
\end{align*}
\end{lemma}

\begin{proof}
By elementary manipulations we obtain
{ \begin{align*}
\mathcal{R}_{k,\epsilon,\psi}(c,t \cdot p^i)& = \epsilon \left( \frac{-c}{\gcd(c,t\cdot p^i)} \right) \overline{\psi} \left( \frac{t\cdot p^i}{\gcd(c,t\cdot p^i)} \right) \left(\frac{\gcd(c,t\cdot p^i)}{t\cdot p^i} \right)^k \\
& = \epsilon \left( \frac{-c/p^{v_c}}{\gcd(c/p^{v_c},t)} \right)  \overline{\psi} \left( \frac{t}{\gcd(c/p^{v_p(c)},t )} \right) \left(\frac{\gcd(c/p^{v_p(c)},t )}{t } \right)^k \\
& \quad \times \epsilon \left( \frac{p^{v_p(c)}}{\gcd(p^{v_p(c)}, p^i)} \right) \overline{\psi} \left( \frac{ p^i}{\gcd(p^{v_p(c)}, p^i)} \right) \left(\frac{\gcd(p^{v_c}, p^i)}{ p^i} \right)^k\\
&=  \epsilon(-1) \mathcal{R}_{k,\epsilon,\psi}(p^{v_p(c)}, p^i) \mathcal{R}_{k,\epsilon,\psi}(c/p^{v_p(c)},t).
\end{align*}}%
Proof of the second equation is similar.

\end{proof}

\begin{theorem} \label{th:3_1}  Let $c,d \mid N$, then
\begin{align*}
& \sum_{t \mid N} \mathcal{S}_{k,N,\epsilon,\psi}(c,t) \mathcal{R}_{k,\epsilon,\psi}(c,t) \mathcal{R}_{k,\epsilon,\psi}(t,d) =\begin{cases}  0 & \mbox{ if $c \neq d$,}\\
 \ds \prod_{p \mid N} \frac{p^k - \epsilon(p) \overline{\psi}(p)}{p^k} & \mbox{ if $c = d$.} \end{cases}
\end{align*}
\end{theorem}

\begin{proof} Let $p \mid N$ be prime and $v_p(N)=v$. Then we use Lemmas \ref{lem:3_1} and \ref{lem:3_2} to obtain 
{ \begin{align}
& \sum_{t \mid N} \mathcal{S}_{k,N,\epsilon,\psi}(c,t) \mathcal{R}_{k,\epsilon,\psi}(c,t) \mathcal{R}_{k,\epsilon,\psi}(t,d)\nonumber \\
& \qquad = \sum_{0 \leq i \leq v}\sum_{t \mid N/p^v} \mathcal{S}_{k,N,\epsilon,\psi}(c,t \cdot p^i) \mathcal{R}_{k,\epsilon,\psi}(c,t \cdot p^i) \mathcal{R}_{k,\epsilon,\psi}(t \cdot p^i,d)\nonumber \\
&\qquad = \sum_{0 \leq i \leq v}\sum_{t \mid N/p^v} \mathcal{S}_{k,p^v,\epsilon,\psi}(p^{v_p(c)},p^i) \mathcal{S}_{k,N/p^v,\epsilon,\psi}(c/p^{v_p(c)}, t)  \nonumber \\
& \qquad \qquad \qquad \times \epsilon(-1) \mathcal{R}_{k,\epsilon,\psi}(p^{v_p(c)}, p^i) \mathcal{R}_{k,\epsilon,\psi}(c/p^{v_p(c)},t)\nonumber \\
& \qquad \qquad \qquad \times \epsilon(-1) \mathcal{R}_{k,\epsilon,\psi}( p^i, p^{v_p(d)}) \mathcal{R}_{k,\epsilon,\psi}(t, d/p^{v_p(d)})\nonumber \\
& \qquad = \sum_{0 \leq i \leq v}  \mathcal{S}_{k,p^v,\epsilon,\psi}(p^{v_p(c)},p^i) \mathcal{R}_{k,\epsilon,\psi}(p^{v_p(c)}, p^i) \mathcal{R}_{k,\epsilon,\psi}( p^i, p^{v_p(d)}) \nonumber \\
&\qquad \qquad  \times  \sum_{t \mid N/p^v}  \mathcal{S}_{k,N/p^v,\epsilon,\psi}(c/p^{v_p(c)},t)  \mathcal{R}_{k,\epsilon,\psi}(c/p^{v_p(c)},t)  \mathcal{R}_{k,\epsilon,\psi}(t, c/p^{v_p(d)}).\nonumber 
\end{align}}%
Using this recursively we obtain
{ \begin{align}
& \sum_{t \mid N} \mathcal{S}_{k,N,\epsilon,\psi}(c,t) \mathcal{R}_{k,\epsilon,\psi}(c,t) \mathcal{R}_{k,\epsilon,\psi}(t,d)\nonumber \\
& \qquad = \prod_{p \mid N}\sum_{0 \leq i \leq v}  \mathcal{S}_{k,p^v,\epsilon,\psi}(p^{v_p(c)},p^i) \mathcal{R}_{k,\epsilon,\psi}(p^{v_p(c)}, p^i) \mathcal{R}_{k,\epsilon,\psi}( p^i, p^{v_p(d)}). \label{7_1}
\end{align}}%

Now we prove for all $p \mid N$ we have
\begin{align}
& \sum_{0 \leq i \leq v}  \mathcal{S}_{k,p^v,\epsilon,\psi}(p^{v_p(c)},p^i) \mathcal{R}_{k,\epsilon,\psi}(p^{v_p(c)}, p^i) \mathcal{R}_{k,\epsilon,\psi}( p^i, p^{v_p(d)}) \nonumber \\
& \qquad =\begin{cases}
 0 & \mbox{ if $v_p(c) \neq v_p(d)$,} \\
\ds  \frac{p^k - \epsilon(p) \overline{\psi}(p)}{p^k} & \mbox{ if $v_p(c) = v_p(d)$.}
 \end{cases} \label{7_2}
\end{align}

We first note that
\bal
\mathcal{S}_{k,p^v,\epsilon,\psi}(p^{v_p(c)},p^i) = \begin{cases}
0 & \mbox{ if $\vert v_p(c) - i \vert >1$,} \\
\ds \frac{p^k + \epsilon(p) \overline{\psi}(p)}{p^k} & \mbox{ if $i = v_p(c)$ and $v> v_p(c) > 0$,}\\
1 & \mbox{ if $i = v_p(c)$ and $v_p(c) = v$,}\\
1 & \mbox{ if $i = v_p(c)$ and $v_p(c) = 0$,}\\
- 1 & \mbox{ if $i = v_p(c) - 1 $ and $v_p(c) > 0$,} \\
- 1 & \mbox{ if $i = v_p(c) + 1 $ and $v_p(c) < v$,}
\end{cases} \label{eqfix_1}
\nal
and
\bal
& \mathcal{R}_{k,\epsilon,\psi}(p^i, p^j) = \begin{cases}
\epsilon(-1) & \mbox{ if $i = j$,}\\
\ds \epsilon(-1) \overline{\psi}(p^{j-i})  \left(\frac{1}{p^{j-i}} \right)^k & \mbox{ if $i < j $,} \\
\epsilon(-p^{i-j}) & \mbox{ if $i > j$.}
\end{cases} \label{eqfix_2}
\nal
The cases
\begin{itemize}
\item[](Case 1) $0 < v_p(c) < v_p(d) \leq v$,
\item[](Case 2) $0 = v_p(c) < v_p(d) \leq v$,
\item[](Case 3) $v > v_p(c) > v_p(d) \geq 0$,
\item[](Case 4) $v = v_p(c) > v_p(d) \geq 0$,
\item[](Case 5) $0 < v_p(c) = v_p(d) < v$,
\item[](Case 6) $0 = v_p(c) = v_p(d)$,
\item[](Case 7) $v= v_p(c) = v_p(d)$,
\end{itemize}
needs to be handled separately, which is done below.

{\bf Case 1:} Let $0 < v_p(c) < v_p(d) \leq v$, then by employing \eqref{eqfix_1} for all $i$ such that $\vert v_p(c) - i \vert >1$ we have
\bals
\mathcal{S}_{k,p^v,\epsilon,\psi}(p^{v_p(c)},p^i) \mathcal{R}_{k,\epsilon,\psi}(p^{v_p(c)}, p^i) \mathcal{R}_{k,\epsilon,\psi}( p^i, p^{v_p(d)}) = 0.
\nals
Therefore we have
\bals
& \sum_{0 \leq i \leq v}  \mathcal{S}_{k,p^v,\epsilon,\psi}(p^{v_p(c)},p^i) \mathcal{R}_{k,\epsilon,\psi}(p^{v_p(c)}, p^i) \mathcal{R}_{k,\epsilon,\psi}( p^i, p^{v_p(d)})\\
&= \mathcal{S}_{k,p^v,\epsilon,\psi}(p^{v_p(c)},p^{v_p(c)-1}) \mathcal{R}_{k,\epsilon,\psi}(p^{v_p(c)}, p^{v_p(c)-1}) \mathcal{R}_{k,\epsilon,\psi}( p^{v_p(c)-1}, p^{v_p(d)})\\
&  \quad + \mathcal{S}_{k,p^v,\epsilon,\psi}(p^{v_p(c)},p^{v_p(c)}) \mathcal{R}_{k,\epsilon,\psi}(p^{v_p(c)}, p^{v_p(c)}) \mathcal{R}_{k,\epsilon,\psi}( p^{v_p(c)}, p^{v_p(d)})\\
& \quad + \mathcal{S}_{k,p^v,\epsilon,\psi}(p^{v_p(c)},p^{v_p(c)+1}) \mathcal{R}_{k,\epsilon,\psi}(p^{v_p(c)}, p^{v_p(c)+1}) \mathcal{R}_{k,\epsilon,\psi}( p^{v_p(c)+1}, p^{v_p(d)}),
\nals
which, by \eqref{eqfix_1} and \eqref{eqfix_2}, equals to
\bals
&= -1 \cdot \epsilon(-p) \cdot \epsilon(-1)\overline{\psi}( p^{v_p(d)-v_p(c)+1}) \left( \frac{1}{p^{v_p(d)-v_p(c)+1}} \right)^k\\
&  \quad +\frac{p^k + \epsilon(p) \overline{\psi}(p)}{p^k} \cdot \epsilon(-1) \cdot \epsilon(-1) \overline{\psi}(p^{v_p(d)-v_p(c)})\left( \frac{1}{p^{v_p(d)-v_p(c)}} \right)^k\\
& \quad + (-1) \cdot \epsilon(-1) \overline{\psi}(p) \left(\frac{1}{p} \right)^k  \cdot \epsilon(-1) \overline{\psi}(p^{v_p(d)-v_p(c)-1})  \left(\frac{1}{p^{v_p(d)-v_p(c)-1}} \right)^k.
\nals
By using multiplicative properties of Dirichlet characters we conclude that this expression is equal to $0$.

{\bf Case 2:} Let $0 = v_p(c) < v_p(d) \leq v$, then by employing \eqref{eqfix_1} and \eqref{eqfix_2} we have
\bals
& \sum_{0 \leq i \leq v}  \mathcal{S}_{k,p^v,\epsilon,\psi}(p^{v_p(c)},p^i) \mathcal{R}_{k,\epsilon,\psi}(p^{v_p(c)}, p^i) \mathcal{R}_{k,\epsilon,\psi}( p^i, p^{v_p(d)})\\
&= \mathcal{S}_{k,p^v,\epsilon,\psi}(1,1) \mathcal{R}_{k,\epsilon,\psi}(1, 1) \mathcal{R}_{k,\epsilon,\psi}( 1, p^{v_p(d)})\\
& \quad + \mathcal{S}_{k,p^v,\epsilon,\psi}(1,p) \mathcal{R}_{k,\epsilon,\psi}(1, p) \mathcal{R}_{k,\epsilon,\psi}( p, p^{v_p(d)})\\
&= \overline{\psi}( p^{v_p(d)}) \left(\frac{1}{p^{v_p(d)}} \right)^k - \overline{\psi}(p^{v_p(d)})  \left(\frac{1}{p^{v_p(d)}} \right)^k,
\nals
which equals to $0$.

{\bf Case 3:} Let $v > v_p(c) > v_p(d) \geq 0$, then by employing \eqref{eqfix_1} and \eqref{eqfix_2} we have
\bals
& \sum_{0 \leq i \leq v}  \mathcal{S}_{k,p^v,\epsilon,\psi}(p^{v_p(c)},p^i) \mathcal{R}_{k,\epsilon,\psi}(p^{v_p(c)}, p^i) \mathcal{R}_{k,\epsilon,\psi}( p^i, p^{v_p(d)})\\
&= \mathcal{S}_{k,p^v,\epsilon,\psi}(p^{v_p(c)},p^{v_p(c)-1}) \mathcal{R}_{k,\epsilon,\psi}(p^{v_p(c)}, p^{v_p(c)-1}) \mathcal{R}_{k,\epsilon,\psi}( p^{v_p(c)-1}, p^{v_p(d)})\\
& \quad + \mathcal{S}_{k,p^v,\epsilon,\psi}(p^{v_p(c)},p^{v_p(c)}) \mathcal{R}_{k,\epsilon,\psi}(p^{v_p(c)}, p^{v_p(c)}) \mathcal{R}_{k,\epsilon,\psi}( p^{v_p(c)}, p^{v_p(d)})\\
& \quad + \mathcal{S}_{k,p^v,\epsilon,\psi}(p^{v_p(c)},p^{v_p(c)+1}) \mathcal{R}_{k,\epsilon,\psi}(p^{v_p(c)}, p^{v_p(c)+1}) \mathcal{R}_{k,\epsilon,\psi}( p^{v_p(c)+1}, p^{v_p(d)})\\
& = - \epsilon(-p) \epsilon( - p^{v_p(c)-1- v_p(d)}) + \frac{p^k + \epsilon(p)\overline{\psi}(p) }{p^k} \cdot \epsilon(-1) \cdot  \epsilon( - p^{v_p(c)- v_p(d)})\\
&  \quad - \epsilon(-1) \overline{\psi}(p) \frac{1}{p^k} \cdot  \epsilon(-p^{v_p(c)+1-v_p(d)}).
\nals
By using multiplicative properties of Dirichlet characters we conclude that this expression is equal to $0$.

{\bf Case 4:} Let $v = v_p(c) > v_p(d) \geq 0$, then by employing \eqref{eqfix_1} and \eqref{eqfix_2} we have
\bals
& \sum_{0 \leq i \leq v}  \mathcal{S}_{k,p^v,\epsilon,\psi}(p^{v_p(c)},p^i) \mathcal{R}_{k,\epsilon,\psi}(p^{v_p(c)}, p^i) \mathcal{R}_{k,\epsilon,\psi}( p^i, p^{v_p(d)})\\
& = \mathcal{S}_{k,p^v,\epsilon,\psi}(p^{v_p(c)},p^{v_p(c)-1}) \mathcal{R}_{k,\epsilon,\psi}(p^{v_p(c)}, p^{v_p(c)-1}) \mathcal{R}_{k,\epsilon,\psi}( p^{v_p(c)-1}, p^{v_p(d)})\\
&  \quad + \mathcal{S}_{k,p^v,\epsilon,\psi}(p^{v_p(c)},p^{v_p(c)}) \mathcal{R}_{k,\epsilon,\psi}(p^{v_p(c)}, p^{v_p(c)}) \mathcal{R}_{k,\epsilon,\psi}( p^{v_p(c)}, p^{v_p(d)})\\
& = - \epsilon(p^{v_p(c)-v_p(d)} )+ \epsilon(p^{v_p(c)-v_p(d)})\\
& = 0.
\nals

{\bf Case 5:} Let $0 < v_p(c) = v_p(d) < v$, then by employing \eqref{eqfix_1}, \eqref{eqfix_2} and multiplicative properties of Dirichlet characters we have
\bals
& \sum_{0 \leq i \leq v}  \mathcal{S}_{k,p^v,\epsilon,\psi}(p^{v_p(c)},p^i) \mathcal{R}_{k,\epsilon,\psi}(p^{v_p(c)}, p^i) \mathcal{R}_{k,\epsilon,\psi}( p^i, p^{v_p(d)})\\
&= \mathcal{S}_{k,p^v,\epsilon,\psi}(p^{v_p(c)},p^{v_p(c)-1}) \mathcal{R}_{k,\epsilon,\psi}(p^{v_p(c)}, p^{v_p(c)-1}) \mathcal{R}_{k,\epsilon,\psi}( p^{v_p(c)-1}, p^{v_p(c)})\\
&  \quad + \mathcal{S}_{k,p^v,\epsilon,\psi}(p^{v_p(c)},p^{v_p(c)}) \mathcal{R}_{k,\epsilon,\psi}(p^{v_p(c)}, p^{v_p(c)}) \mathcal{R}_{k,\epsilon,\psi}( p^{v_p(c)}, p^{v_p(c)})\\
& \quad + \mathcal{S}_{k,p^v,\epsilon,\psi}(p^{v_p(c)},p^{v_p(c)+1}) \mathcal{R}_{k,\epsilon,\psi}(p^{v_p(c)}, p^{v_p(c)+1}) \mathcal{R}_{k,\epsilon,\psi}( p^{v_p(c)+1}, p^{v_p(c)})\\
&= - \epsilon(-p) \cdot \epsilon(-1) \overline{\psi}(p) \frac{1}{p^k}+  \frac{p^k + \epsilon(p) \overline{\psi}(p)}{p^k} \cdot \epsilon(-1)  \cdot \epsilon(-1) - \epsilon(-1) \overline{\psi}(p) \frac{1}{p^k} \cdot \epsilon(-p)\\
&= \frac{p^k - \epsilon(p) \overline{\psi}(p)}{p^k}.
\nals

{\bf Case 6:} Let $0=v_p(c) = v_p(d) $, then by employing \eqref{eqfix_1}, \eqref{eqfix_2} and multiplicative properties of Dirichlet characters we have
\bals
& \sum_{0 \leq i \leq v}  \mathcal{S}_{k,p^v,\epsilon,\psi}(p^{v_p(c)},p^i) \mathcal{R}_{k,\epsilon,\psi}(p^{v_p(c)}, p^i) \mathcal{R}_{k,\epsilon,\psi}( p^i, p^{v_p(d)})\\
&=  \mathcal{S}_{k,p^v,\epsilon,\psi}(1,1) \mathcal{R}_{k,\epsilon,\psi}(1,1) \mathcal{R}_{k,\epsilon,\psi}( 1,1 )+ \mathcal{S}_{k,p^v,\epsilon,\psi}(1,p ) \mathcal{R}_{k,\epsilon,\psi}(1, p ) \mathcal{R}_{k,\epsilon,\psi}( p, 1)\\
& =  1 \cdot \epsilon(-1) \cdot \epsilon(-1) - \epsilon(-1) \overline{\psi}(p) \frac{1}{p^k} \cdot \epsilon(-p)\\
& = \frac{p^k - \epsilon(p) \overline{\psi}(p)}{p^k}.
\nals

{\bf Case 7:} Let $v=v_p(c) = v_p(d)$, then by employing \eqref{eqfix_1}, \eqref{eqfix_2} and multiplicative properties of Dirichlet characters we have
\bals
& \sum_{0 \leq i \leq v}  \mathcal{S}_{k,p^v,\epsilon,\psi}(p^{v_p(c)},p^i) \mathcal{R}_{k,\epsilon,\psi}(p^{v_p(c)}, p^i) \mathcal{R}_{k,\epsilon,\psi}( p^i, p^{v_p(d)})\\
&= \mathcal{S}_{k,p^v,\epsilon,\psi}(p^{v},p^{v-1}) \mathcal{R}_{k,\epsilon,\psi}(p^{v}, p^{v-1}) \mathcal{R}_{k,\epsilon,\psi}( p^{v-1}, p^{v})\\
&  \quad + \mathcal{S}_{k,p^v,\epsilon,\psi}(p^{v},p^{v}) \mathcal{R}_{k,\epsilon,\psi}(p^{v}, p^{v}) \mathcal{R}_{k,\epsilon,\psi}( p^{v}, p^{v})\\
&= -1 \cdot  \epsilon(-p) \cdot \epsilon(-1) \overline{\psi}(p) \frac{1}{p^k} + 1 \cdot \epsilon(-1)  \cdot \epsilon(-1) \\
& = \frac{p^k - \epsilon(p) \overline{\psi}(p)}{p^k}.
\nals

Finally, if $c \neq d$, then there exists a prime $p \mid N$ such that $v_p(c) \neq v_p(d)$. Hence by \eqref{7_2} the product in \eqref{7_1} is $0$. If $c=d$ then for all prime divisors $p$ of $N$ we have $v_p(c) = v_p(d)$. Therefore by \eqref{7_1} and \eqref{7_2} we have the desired result.


\end{proof}

\section{Constant terms of expansions of Eisenstein series at the cusps} \label{sec_ft}
Recall that $E_{k}(\epsilon,\psi;dz)$ is defined by \eqref{eis} and we have
\begin{align*}
& {E_{k}(\epsilon,\psi;dz )}\in E_{k}(\Gamma_0(N),\chi) \mbox{ when $(k,\epsilon,\psi)\neq (2,\chi_1,\chi_1)$},
\end{align*}
and
\begin{align*}
& L_d(z) :=E_{2}(\chi_1,\chi_1;z)-dE_{2}(\chi_1,\chi_1;dz) \in E_{2}(\Gamma_0(N),\chi_1).
\end{align*}
The constant terms of Eisenstein series in the expansion at the cusp $a/c$ with $\gcd(a,c)=1$ are given by
\begin{align}
& [0]_{a/c} E_{k}(\epsilon,\psi;dz ) = \ds \overline{\psi}(a)\mathcal{R}_{k,\epsilon,\psi} (c,Md)  \mbox{ when $(k,\epsilon,\psi)\neq (2,\chi_1,\chi_1)$ and} \label{7_6} \\
& [0]_{a/c} L_d(z)= \mathcal{R}_{2,\chi_1,\chi_1}(c,1)-d\mathcal{R}_{2,\chi_1,\chi_1}(c,d), \label{7_4}
\end{align}
where $\mathcal{R}_{k,\epsilon,\psi} (c,t) $ is defined by \eqref{RR}. For \eqref{7_6} see \cite[(6.2)]{sqfreepaper}, \cite[Proposition 8.5.6 and Ex. 8.7 (i) on pg. 308]{cohenbook}. The formula \eqref{7_4} is proved later in this section.  

The structure of the terms $[0]_{a/c}E_k(\epsilon,\psi;dz)$ is complicated and difficult to work with. We observe that taking the average $[0]_{c,{\psi}} E_k( \epsilon,\psi;dz)$ gives constant terms a very nice structure which is easier to work with, see \eqref{fix2_1}. Throughout the section we assume $k,N \in \nn$, $\epsilon$ and $\psi$ are primitive Dirichlet characters with conductors $L$ and $M$, respectively, such that $LM \mid N$ and $(k,\epsilon,\psi) \neq (2,\chi_1,\chi_1)$.
\begin{lemma} \label{lem:4_1} Let $c \mid N$ and $LMd \mid N$. If $M \nmid c$, or $M \mid c$ and $L \nmid N/c$, then
\begin{align*}
[0]_{a/c} E_k(\epsilon,\psi;dz) =0.
\end{align*}
\end{lemma}

\begin{proof}
First we let $M \nmid c$. Then $ M \nmid \gcd(Md,c) $. Thus $\ds \gcd \left( \frac{Md}{\gcd(Md,c)},M \right) \mid M$, which implies $\ds  \overline{\psi} \left( \frac{Md}{\gcd(Md,c)} \right)=0$ since the conductor of $\psi$ is $M$. Therefore the result follows from \eqref{7_6}.\\

Second, we let $M \mid c$, $c_1=M/c$, thus $c_1 \mid N/M$. Assume $L \nmid N/c$, thus $c_1 \nmid N/LM$. Since $\frac{(N/M)}{L} \in \zz$, $\frac{(N/M)}{c_1} \in \zz$ and $\frac{(N/M)/c_1}{L} \not\in \zz$, we have $\gcd(c_1,L) \neq 1$. Additionally, there exists a prime $p$ dividing $c_1$ such that 
\begin{align}
& v_p(c_1) > v_p(N) - v_p(M)-v_p(L). \label{5_1}
\end{align}
Since $c_1 \mid N/M$, for all $p \mid c_1$ we have
\begin{align}
& v_p(c_1) \leq v_p(N) - v_p(M). \label{JJ5_1}
\end{align}
By \eqref{5_1} and \eqref{JJ5_1} we have
\begin{align*}
& v_p(N) - v_p(M)> v_p(N) - v_p(M)-v_p(L). 
\end{align*}
Therefore 
\begin{align}
& v_p(L)>  0.\label{5_2}
\end{align}
Since $d \mid N/LM$ we have
\begin{align}
v_p(N)-v_p(M)-v_p(L) \geq v_p(d) \geq 0. \label{5_3}
\end{align}
Inequalities \eqref{5_1} and \eqref{5_3} together implies $v_p(c_1) > v_p(d)$. Therefore by employing \eqref{5_2} we have
\begin{align*}
p \mid \gcd\left( \frac{p^{v_p(c_1)}}{\gcd(p^{v_p(c_1)},p^{v_p(d)})},p^{v_p(L)} \right).
\end{align*}
That is,
\begin{align*}
p \mid \gcd\left( \frac{c_1}{\gcd(c_1,d)},L \right).
\end{align*}
This implies $\ds  \epsilon \left( \frac{c}{\gcd(Md,c)} \right)=0$ since conductor of $\epsilon$ is $L$. Therefore the result follows from \eqref{7_6}.

\end{proof}

\begin{lemma} \label{lem:4_2} Let $c \in \nn$, and let $\psi_1,\psi_2$ be two primitive Dirichlet characters with conductors $M_1$ and $M_2$, respectively. Let both $M_1$ and $M_2$ divide $c$. Then we have
\bals
\sum_{\substack{a =1, \\ \gcd(a,c)=1}}^c \overline{\psi_1} (a) {\psi_2}(a) = \begin{cases}
0 & \mbox{if $\psi_1 \neq \psi_2$,}\\
\ds  \phi(c) & \mbox{if $\psi_1 = \psi_2$.}
\end{cases}
\nals
\end{lemma}

\begin{proof}
The modulus of the character $\overline{\psi_1} {\psi_2}$ is $\lcm(M_1,M_2)$, see \cite[pg. 80]{miyake}. We let $A=\lcm(M_1,M_2)$, then $A \mid c$. We have
\bals
\sum_{\substack{a =1, \\ \gcd(a,c)=1}}^c \overline{\psi_1} (a) {\psi_2}(a) & = \sum_{a=1}^c \overline{\psi_1} (a) {\psi_2}(a) \sum_{s \mid \gcd(a,c)} \mu(s)\\
& = \sum_{s \mid c} \mu(s) \sum_{\substack{a =1, \\ s \mid a}}^c \overline{\psi_1} (a) {\psi_2}(a)\\
& = \sum_{s \mid c} \mu(s) \overline{\psi_1} (s) {\psi_2}(s) \sum_{\substack{t=1}}^{c/s} \overline{\psi_1} (t) {\psi_2}(t).
\nals
We have $\overline{\psi_1} (s) {\psi_2}(s)=0$ whenever $\gcd(A,s)>1$, therefore we obtain
\bal
\sum_{\substack{a =1, \\ \gcd(a,c)=1}}^c \overline{\psi_1} (a) {\psi_2}(a)& = \sum_{\substack{s \mid c,\\ \gcd(A,s)=1}} \mu(s) \overline{\psi_1} (s) {\psi_2}(s) \sum_{\substack{t=1}}^{c/s} \overline{\psi_1} (t) {\psi_2}(t). \label{eq223_1}
\nal
Now if $A \mid c$, $s \mid c$ and $\gcd(A,s)=1$, then $A \mid c/s$. First we consider $\psi_1 \neq \psi_2$. Then
\bal
\sum_{\substack{t=1}}^{c/s} \overline{\psi_1} (t) {\psi_2}(t)=0. \label{eq223_2}
\nal
Therefore if $\psi_1 \neq \psi_2$ by \eqref{eq223_1} and \eqref{eq223_2} we have
\bals
\sum_{\substack{a =1, \\ \gcd(a,c)=1}}^c \overline{\psi_1} (a) {\psi_2}(a)=0.
\nals
Second we consider the case when $\psi_1 = \psi_2$. Noting that in this case $A=M_2$, we have
\bal
\sum_{\substack{t=1}}^{c/s} \overline{\psi_1} (t) {\psi_2}(t)= \frac{c}{sM_2} \phi(M_2). \label{eq223_3}
\nal
Therefore if $\psi_1 = \psi_2$ by by \eqref{eq223_1} and \eqref{eq223_3} we have
\bal
\sum_{\substack{a =1, \\ \gcd(a,c)=1}}^c \overline{\psi_1} (a) {\psi_2}(a) & = \sum_{\substack{s \mid c, \\ \gcd(M_2,s)=1}} \mu(s) \frac{c}{sM_2} \phi(M_2) = c \cdot \frac{\phi(M_2)}{M_2} \sum_{\substack{s \mid c, \\ \gcd(M_2,s)=1}}  \frac{\mu(s)}{s}. \label{eq21_1}
\nal
Noting that
\bals
\sum_{\substack{s \mid c}} \frac{\mu(s)}{s}  = \frac{\phi(c)}{c}= \prod_{p \mid c} \frac{p-1}{p},
\nals
we have
\bals
\sum_{\substack{s \mid c, \\ \gcd(M_2,s)=1}} \frac{\mu(s)}{s}  =  \prod_{\substack{p \mid c,\\ \gcd(p,M_2)=1}} \frac{p-1}{p}= \frac{\prod_{\substack{p \mid c}} \frac{p-1}{p}}{\prod_{\substack{p \mid M_2}} \frac{p-1}{p}} = \frac{\phi(c)/c}{\phi(M_2)/M_2}.
\nals
Putting this in \eqref{eq21_1} completes the proof.

\end{proof}
Before we prove the main result of this section we prove \eqref{7_4}.
\begin{lemma} \label{7_5} Let $\gcd(a,c)=1$, then we have
\begin{align*}
[0]_{a/c} L_d(z)= \mathcal{R}_{2,\chi_1,\chi_1}(c,1)-d\mathcal{R}_{2,\chi_1,\chi_1}(c,d).
\end{align*}
\end{lemma}

{\begin{proof}
Since $\gcd(a,c)=1$, there exist $\beta, \gamma \in \zz$ such that $A=\begin{bmatrix} a & \beta \\ c & \gamma \end{bmatrix} \in SL_2(\zz)$. Then by \cite[(1.21)]{Kohler} we have
\begin{align}
& E_2(\chi_1,\chi_1;A(z))=(cz+\gamma)^2 E_2(\chi_1,\chi_1;z) - \frac{6ic}{ \pi} (cz+\gamma), \label{7_7}
\end{align}
where $A(z)$ is the usual linear fractional transformation. Let $e= \frac{a d  }{\gcd(c,a d  )}$ and $g=\frac{c}{\gcd(c,a d  )}$. Then since $\gcd(e,g)=1$ there exist $f, h$ such that $ \begin{bmatrix} e & f \\ g & h \end{bmatrix}  \in SL_2(Z)$. Hence we have
\begin{align*}
 E_{2}(\chi_1,\chi_1; d   A(z);) &=  E_{2} \left(\chi_1,\chi_1;\begin{bmatrix} e & f \\ g & h \end{bmatrix} \begin{bmatrix} a h  d   - c f & \beta h d   - \gamma f \\ - ag d   + ce & -\beta g d   + \gamma e \end{bmatrix} (z) \right)\\
& =  E_{2} \left( \chi_1,\chi_1;\begin{bmatrix} \frac{a d  }{\gcd(c,a d  )} & f \\ \frac{c}{\gcd(c,a d  )} & h \end{bmatrix} \begin{bmatrix} a h  d   - c f & \beta h d   - \gamma f \\ 0 &  \frac{ d }{\gcd(c,a d )}  \end{bmatrix} (z) \right)\\
& =  \left(\frac{\gcd(c  ,d)}{ d  } \right)^2 (cz+\gamma)^2 E_{2} \left(\chi_1,\chi_1; \begin{bmatrix} a h  d   - c f & \beta h d   - \gamma f \\ 0 &  \frac{d}{\gcd(c,a d  )}  \end{bmatrix} (z) \right)\\
& \qquad  - \frac{6ic}{ \pi  d  } (cz+\gamma),
\end{align*}
where in the last line we used \eqref{7_7}. Thus we obtain
\begin{align*}
[0]_{a/c} L_d(z) & =[0]_{a/c} (E_2(\chi_1,\chi_1;z)- d  E_2(\chi_1,\chi_1;d  z)) = \frac{ d -\gcd( c ,d)^2}{ d }\\
& = \mathcal{R}_{2,\chi_1,\chi_1}(c,1)-d\mathcal{R}_{2,\chi_1,\chi_1}(c,d).
\end{align*}

\end{proof}}%

\begin{theorem} \label{th:4_1} Let $c \mid N$ and let $(\epsilon_1,\psi_1),(\epsilon_2,\psi_2) \in \{ (\epsilon,\psi) \in \mathcal{E}(k,N,\chi) : M \mid c \}$.  Then we have
\begin{align*}
[0]_{c,{\psi_2}} E_k( \epsilon_1,\psi_1;dz)=\begin{cases}
[0]_{1/c} E_k(\epsilon_2,\psi_2;dz) & \mbox{ if $\psi_1=\psi_2$, }\\
0 & \mbox{otherwise.}
\end{cases} 
\end{align*}
Let $c \mid N$ and let $(\epsilon_2,\psi_2) \in \{ (\epsilon,\psi) \in \mathcal{E}(2,N,\chi_1) : M \mid c \}$, then we have
\begin{align*}
[0]_{c,{\psi_2}} L_d(z)=\begin{cases}
[0]_{1/c} L_d(z) & \mbox{ if $\psi_2=\chi_1$, }\\
0 & \mbox{otherwise.}
\end{cases} 
\end{align*}
\end{theorem}
\begin{proof} If $(k, \epsilon, \psi) \neq (2,\chi_1,\chi_1)$ by \eqref{7_6} we have
{ \begin{align*}
 [0]_{c,{\psi_2}} E(\epsilon_1,\psi_1;dz)
& = \frac{1}{\phi(c)}  \sum_{\substack{a =1,\\ \gcd(a,c)=1} }^c  {\psi_2}(a) \overline{\psi_1}(a) \mathcal{R}_{k,\epsilon_1,\psi_1}(c,M_1d) \\
& = [0]_{1/c} E_k(\epsilon_1,\psi_1;dz) \frac{1}{\phi(c)} \sum_{\substack{a =1,\\ \gcd(a,c)=1} }^c  {\psi_2}(a)  \overline{\psi_1}(a).
\end{align*}}%
Therefore by Lemma \ref{lem:4_2} we obtain the the first part of the statement. Proof of the second part is similar.

\end{proof}

\section{Proof of the Main Theorem} \label{sec_proof}

Recall that $E_{k}(\epsilon,\psi;dz)$ is defined by \eqref{eis} and the set
\begin{align}
& \{ E_{k}(\epsilon,\psi;dz ) : (\epsilon,\psi) \in \mathcal{E}(k,N,\chi), d \mid N/LM \}\label{basis}
\end{align}
constitutes a basis for $E_k(\Gamma_0(N),\chi)$ whenever $(k,\chi) \neq (2,\chi_1)$ and the set
\begin{align}
& \{ E_{2}(\chi_1,\chi_1;z)-d E_{2}(\chi_1,\chi_1;dz) : 1< d \mid N/LM \} \label{basis2}\\
& \qquad \cup \{ E_{2}(\epsilon,\psi;dz ) : (\epsilon,\psi) \in \mathcal{E}(2,N,\chi_1), (\epsilon,\psi) \neq ( \chi_1,\chi_1), d \mid N/LM \} \nonumber
\end{align}
constitutes a basis for  $E_2(\Gamma_0(N),\chi_1)$, see \cite[Theorems 8.5.17 and 8.5.22]{cohenbook}, or \cite[Proposition 5]{weis}. 

Now we prove the main theorem whenever $(k,\chi) \neq (2,\chi_1)$. Let $f(z) \in M_k(\Gamma_0(N),\chi)$ where $N,k \in \nn$, $k \geq 2$ and $(k,\chi) \neq (2,\chi_1)$. Then by \eqref{basis} we have
\begin{align}
E_f(z) = \sum_{(\epsilon, \psi) \in \mathcal{E}(k,N,\chi) } \sum_{d \mid N/LM} a_f(\epsilon,\psi,d) E_k(\epsilon,\psi; dz), \label{deceq_1}
\end{align}
for some $a_f(\epsilon,\psi,d) \in \cc$. Our strategy for the proof is, using the interplay between the constant terms of Eisenstein series, to create sets of linear equations (see \eqref{fix2_1}) and to solve those sets of linear equations for $a_f(\epsilon,\psi,d)$ using Theorem \ref{th:3_1}.

By \eqref{decomp} we have $f(z)=E_f(z)+S_f(z)$, where $E_f(z) \in E_k(\Gamma_0(N),\chi)$ and $S_f(z) \in S_k(\Gamma_0(N),\chi)$ are unique. Since by definition $S_f(z)$ vanishes at all cusps, we have $[0]_{a/c}f(z)=[0]_{a/c}E_f(z)$. Therefore by \eqref{deceq_1} for each $c \mid N$ and $a \in \zz$ such that $\gcd(a,c)=1$, we obtain
\begin{align*}
[0]_{a/c}f(z) = \sum_{(\epsilon, \psi) \in \mathcal{E}(k,N,\chi) } \sum_{d \mid N/LM} a_f(\epsilon,\psi,d) [0]_{a/c} E_k(\epsilon,\psi;dz).
\end{align*}
Let $(\epsilon_2,\psi_2) \in \mathcal{E}(k,N,\chi)$, and let the conductors of $\epsilon_2$ and $\psi_2$ be $L_2$ and $M_2$, respectively. Note that for each $\psi_2$ there is a unique $\epsilon_2$ such that $(\epsilon_2,\psi_2) \in \mathcal{E}(k,N,\chi)$. If we average the constant terms with $\psi_2$ using \eqref{newcusp}, then for all $c \mid N$ we obtain
\begin{align*}
[0]_{c, {\psi_2}}f(z) = \sum_{(\epsilon, \psi) \in \mathcal{E}(k,N,\chi) } \sum_{d \mid N/LM} a_f(\epsilon,\psi,d) [0]_{c, {\psi_2}} E_k(\epsilon,\psi;dz).
\end{align*}
Our goal here is to isolate a set of linear equations from which we can determine $a_f(\epsilon_2,\psi_2,d)$ for all $d \mid N/L_2M_2$. By Lemma \ref{lem:4_1} we have $[0]_{c, {\psi_2}} E_k(\epsilon_2,\psi_2;dz)=0$ if $c \mid N$ is such that $M_2 \mid c$, or $M_2 \nmid c$ and $L_2 \mid N/c$. Therefore from now on we restrict $c$ to be in $C_{N}(\epsilon_2,\psi_2)$, see \eqref{CMdef} for definition. By applying Lemma \ref{lem:4_1} one more time we have $[0]_{c,\psi_2} E_k(\epsilon,\psi;dz)=0$ if $M \nmid c$. Therefore for all $c \in  C_{N}(\epsilon_2,\psi_2)$ we have
\begin{align}
[0]_{c, {\psi_2}}f(z) = \sum_{\substack{(\epsilon, \psi) \in \mathcal{E}(k,N,\chi),\\ M \mid c }} \sum_{d \mid N/LM} a_f(\epsilon,\psi,d) [0]_{c, {\psi_2}} E_k(\epsilon,\psi;dz). \label{fix2_2}
\end{align}
Recall that for each $\psi_2$ there is a unique $(\epsilon_2,\psi_2) \in \mathcal{E}(k,N,\chi)$. Additionally, for all $c \in  C_{N}(\epsilon_2,\psi_2)$ we have $(\epsilon_2,\psi_2) \in \{(\epsilon,\psi) \in \mathcal{E}(k,N,\chi) : M \mid c \}$. Therefore for all $c \in  C_{N}(\epsilon_2,\psi_2)$ we have
\begin{align*}
[0]_{c, {\psi_2}}f(z) &= \sum_{\substack{(\epsilon, \psi) \in \mathcal{E}(k,N,\chi),\\ (\epsilon, \psi) \neq (\epsilon_2,\psi_2) \\ M \mid c }} \sum_{d \mid N/LM} a_f(\epsilon,\psi,d) [0]_{c, {\psi_2}} E_k(\epsilon,\psi;dz)\\
& \quad + \sum_{d \mid N/L_2M_2} a_f(\epsilon_2,\psi_2,d) [0]_{c, {\psi_2}} E_k(\epsilon_2,\psi_2;dz).
\end{align*}
From this, using Theorem \ref{th:4_1}, we obtain
\begin{align*}
[0]_{c, {\psi_2}}f(z) &= \sum_{d \mid N/L_2M_2} a_f(\epsilon_2,\psi_2,d) [0]_{1/c} E_k(\epsilon_2,\psi_2;dz).
\end{align*}
Since $M_2 \mid c$ we have
\begin{align*}
[0]_{1/c} E_k(\epsilon_2,\psi_2;dz)=\mathcal{R}_{k,\epsilon_2,\psi_2}(c,M_2d)= \mathcal{R}_{k,\epsilon_2,\psi_2}(c/M_2,d).
\end{align*}
Hence for all $c \in  C_{N}(\epsilon_2,\psi_2)$ we have
\begin{align}
[0]_{c, {\psi_2}}f(z) = \sum_{d \mid N/L_2M_2} a_f(\epsilon_2,\psi_2,d) \mathcal{R}_{k,\epsilon_2,\psi_2}(c/M_2,d). \label{fix2_1}
\end{align}
Below we solve the equations coming from \eqref{fix2_1} for $a_f(\epsilon_2,\psi_2,d)$ using Theorem \ref{th:3_1}. For $d_2 \mid N/L_2M_2$ we consider the sum
\begin{align}
\sum_{c \in C_{N}(\epsilon_2,\psi_2) } \mathcal{R}_{k,\epsilon_2,\psi_2}(d_2,c/M_2) \mathcal{S}_{k,N/L_2M_2,\epsilon_2,\psi_2}(d_2,c/M_2) [0]_{c, {\psi_2}}f(z), \label{fix28_3}
\end{align}
which, by \eqref{fix2_1}, equals to
{\footnotesize \begin{align}
 = \sum_{c \in C_{N}(\epsilon_2,\psi_2) } \mathcal{R}_{k,\epsilon_2,\psi_2}(d_2,c/M_2) \mathcal{S}_{k,N/L_2M_2,\epsilon_2,\psi_2}(d_2,c/M_2) \sum_{d \mid N/L_2M_2} a_f(\epsilon_2,\psi_2,d) \mathcal{R}_{k,\epsilon_2,\psi_2}(c/M_2,d). \label{fix28_1}
\end{align}}
Rearranging the terms of \eqref{fix28_1} we obtain
{\footnotesize \begin{align}
& \sum_{c \in C_{N}(\epsilon_2,\psi_2) } \mathcal{R}_{k,\epsilon_2,\psi_2}(d_2,c/M_2) \mathcal{S}_{k,N/L_2M_2,\epsilon_2,\psi_2}(d_2,c/M_2) [0]_{c, {\psi_2}}f(z) \nonumber\\
& =\sum_{d \mid N/L_2M_2} a_f(\epsilon_2,\psi_2,d) \sum_{c \in C_{N}(\epsilon_2,\psi_2) } \mathcal{R}_{k,\epsilon_2,\psi_2}(d_2,c/M_2) \mathcal{S}_{k,N/L_2M_2,\epsilon_2,\psi_2}(d_2,c/M_2)   \mathcal{R}_{k,\epsilon_2,\psi_2}(c/M_2,d). \label{eq222_1}
\end{align}}
Recall that $C_{N}(\epsilon_2,\psi_2)$ is defined by \eqref{CMdef} and is a set equivalent to the set
\bals
\{ c : M_2 \mid c, ~c/M_2 \mid N/L_2M_2 \},
\nals
i.e., $c/M_2$ runs through all the divisors of $ N/L_2M_2$ as $c$ runs through all the elements of $C_{N}(\epsilon_2,\psi_2)$. In Theorem \ref{th:3_1} we use this and we replace $N$ by $N/L_2M_2$, $t$ by $c/M_2$, $c$ by $d_2$ and $d$ by $d$  to obtain
\begin{align}
& \sum_{c \in C_{N}(\epsilon_2,\psi_2) } \mathcal{R}_{k,\epsilon_2,\psi_2}(d_2,c/M_2) \mathcal{S}_{k,N/L_2M_2,\epsilon_2,\psi_2}(d_2,c/M_2)   \mathcal{R}_{k,\epsilon_2,\psi_2}(c/M_2,d) \nonumber \\
& = \begin{cases}
\ds \prod_{ p \mid N/L_2M_2} \frac{p^k - \epsilon_2(p) \overline{\psi_2}(p)}{p^k} & \mbox{ if $d=d_2$,}\\
0 & \mbox{ if $d \neq d_2$.} \label{eq222_2}
 \end{cases} 
\end{align}
Therefore from \eqref{eq222_1} and \eqref{eq222_2} we obtain
\begin{align*}
& \sum_{c \in C_{N}(\epsilon_2,\psi_2) } \mathcal{R}_{k,\epsilon_2,\psi_2}(d_2,c/M_2) \mathcal{S}_{k,N/L_2M_2,\epsilon_2,\psi_2}(d_2,c/M_2) [0]_{c, {\psi_2}}f(z) \\
 & =  a_f(\epsilon_2,\psi_2,d_2) \prod_{ p \mid N/L_2M_2} \frac{p^k - \epsilon_2(p) \overline{\psi_2}(p)}{p^k}. 
\end{align*}
Since $p \mid L_2 M_2$ implies $\epsilon_2(p) \overline{\psi_2}(p) =0$ we have
{\footnotesize \begin{align*}
a_f(\epsilon_2,\psi_2,d_2) &= \prod_{p \mid N} \frac{p^k}{p^k-\epsilon_2(p) \overline{\psi_2}(p)} \sum_{c \in C_{N}(\epsilon_2,\psi_2) } \mathcal{R}_{k,\epsilon_2,\psi_2}(d_2,c/M_2) \mathcal{S}_{k,N/L_2M_2,\epsilon_2,\psi_2}(d_2,c/M_2) [0]_{c,{\psi_2}}f.
\end{align*}}
This completes the proof of Theorem \ref{mainth} when $(k,\chi) \neq (2, \chi_1)$.  

Now let $(k,\chi) = (2, \chi_1)$, then a basis of $E_2(\Gamma_0(N),\chi_1)$ is given by \eqref{basis2}. Using Lemma \ref{7_5}, Theorem \ref{th:4_1} and arguments similar to the first part of this proof we obtain
\begin{align*}
E_f(z) & =\sum_{1<d \mid N} c_f(\chi_1,\chi_1,d) L_d(z) +\sum_{\substack{(\epsilon, \psi) \in \mathcal{E}(2,N,\chi), \\ (\epsilon,\psi) \neq (\chi_1,\chi_1)} } \sum_{d \mid N/LM} a_f(\epsilon,\psi,d) E_2(\epsilon,\psi;dz),
\end{align*}
where $a_f(\epsilon,\psi,d)$ is as above (with $k=2$) and 
{\footnotesize \begin{align*}
c_f(\chi_1,\chi_1,d) = -\frac{1}{d} \prod_{p \mid N} \frac{p^2}{p^2 - 1 } \sum_{\substack{c \mid N} }  \mathcal{R}_{2,\chi_1,\chi_1}(d,c) \mathcal{S}_{2,N,\chi_1,\chi_1}(d,c) [0]_{c,{\chi_1}}f = -\frac{1}{d} a_f(\chi_1,\chi_1,d).
\end{align*}}
On the other hand we have
{ \begin{align*}
\sum_{1<d \mid N} c_f(\chi_1,\chi_1,d) L_d(z) & = \sum_{1<d \mid N} c_f(\chi_1,\chi_1,d) (E_2(\chi_1,\chi_1;z)-dE_2(\chi_1,\chi_1;dz))\\
&=\sum_{1<d \mid N} c_f(\chi_1,\chi_1,d) E_2(\chi_1,\chi_1;z) \\
& \qquad + \sum_{1<d \mid N} a_f(\chi_1,\chi_1,d) E_2(\chi_1,\chi_1;dz)\\
&= \sum_{d \mid N} a_f(\chi_1,\chi_1,d) E_2(\chi_1,\chi_1;dz),
\end{align*}}%
since
\begin{align*}
a_f(\chi_1,\chi_1,1) \prod_{p \mid N} \frac{p^2-1}{p^2}  &=\sum_{c \mid N}  \mathcal{R}_{2,\chi_1,\chi_1}(1,c)\mathcal{S}_{2,N,\chi_1,\chi_1}(1,c) [0]_{c,\chi_1} f \\
&= \sum_{c \mid N} \frac{\mu(c)}{c^2} \sum_{1< d \mid N} c_f(\chi_1,\chi_1,d) \frac{d-\gcd(d,c)^2}{d} \\
& =\sum_{1< d \mid N} c_f(\chi_1,\chi_1,d)  \sum_{c \mid N} \frac{\mu(c)}{c^2}  \frac{d-\gcd(d,c)^2}{d} \\
& = \prod_{p \mid N} \frac{p^2-1}{p^2} \sum_{1< d \mid N} c_f(\chi_1,\chi_1,d),
\end{align*}
i.e., $\sum_{1< d \mid N} c_f(\chi_1,\chi_1,d)=a_f(\chi_1,\chi_1,1)$. This completes the proof of the Main Theorem.

At last we prove a lemma which is useful in reducing the number of constant term computations in applications of Theorem \ref{mainth}.
\begin{lemma} \label{reducer} Let $f(z) \in M_k(\Gamma_0(N),\chi)$ and $c \mid N$. Let $a/c$ and $a'/c$ be equivalent cusps of $\Gamma_0(N)$. If $(\epsilon,\psi) \in \mathcal{E}(k,N,\chi)$ with $M \mid c$ then we have
\begin{align*}
{\psi}(a)[0]_{a/c}f={\psi}(a')[0]_{a'/c}f.
\end{align*}
\end{lemma}

\begin{proof}
Let $a/c$ and $a'/c$ be equivalent cusps of $\Gamma_0(N)$, then there exists a matrix $\begin{bmatrix} \alpha & \beta\\ \gamma & \delta  \end{bmatrix} \in \Gamma_0(N)$ such that
\bal
\begin{bmatrix} \alpha & \beta\\ \gamma & \delta  \end{bmatrix} \begin{bmatrix} a & b \\ c & d  \end{bmatrix} = \begin{bmatrix} a' & b' \\ c & d'  \end{bmatrix}. \label{eqfeb_1}
\nal
Then using transformation properties of modular forms we have
\bals
\psi(a') [0]_{a'/c} f & = \psi(a') \lim_{z \rightarrow i\infty} (cz+d')^{-k} f\left( \frac{a' z + b'}{cz + d'} \right)\\
& = \psi(a') \lim_{z \rightarrow i\infty} (cz+d')^{-k} \chi(\delta) \left( \gamma  \frac{a z + b}{cz + d} + \delta \right)^k f\left( \frac{a z + b}{cz + d} \right)\\
& = \psi(a') \chi(\delta) \lim_{z \rightarrow i\infty} (cz+d)^{-k} f\left( \frac{a z + b}{cz + d} \right)\\
& = \psi(a') \chi(\delta) [0]_{a/c}f.
\nals
We have $M\mid c$ and by \eqref{eqfeb_1} we have $a'=\alpha a + \beta c$, thus $\psi(a')=\psi(\alpha )\psi( a)$. Since $M \mid c$, $c \mid N$ and $N \mid \gamma$ we have $M \mid \gamma$, therefore we have $ 1= \psi(1)= \psi( \alpha \delta - \gamma \beta)$ which implies $\psi(\alpha) = \overline{\psi}(\delta)$. Putting these together we obtain
\bals
\psi(a') \chi(\delta)= \psi(a) \overline{\psi}(\delta) \chi(\delta).
\nals
Since $\gcd(\delta,N)=1$ we have $\overline{\psi}(\delta) \chi(\delta)=\epsilon(\delta)$. Now we prove $\epsilon(\delta)=1$ which finishes the proof. Recall that $LM \mid N$, therefore $c \mid M$ implies $L \mid N/c$, i.e., $L \mid \gamma/c$. From \eqref{eqfeb_1} we have $\delta = 1 - a \gamma/c$, thus, since $\gcd(a,c)=1$ and $L \mid \gamma/c$, we have $\epsilon(\delta)=\epsilon(1 - a \gamma/c)=\epsilon(1)=1$. 
\end{proof}

\section*{Acknowledgements}
I would like to thank Professor Amir Akbary for helpful discussions throughout the course of this research. I am also grateful to Professor Shaun Cooper, who gave the vision which initiated this research. Words cannot adequately express my gratitude towards Professor Emeritus Kenneth S. Williams, who has given many useful suggestions on an earlier version of this manuscript. I would like to thank the referee for pointing out the problems in an earlier version of the proof of Theorem \ref{th:3_1}.

\appendix

\section{The SAGE functions for computing the constant terms of eta quotients at a given cusp} \label{appa}

Let $r_d \in \zz$, not all zeros,  $N \in \nn$ and define
\bals
f(z)=\prod_{d \mid N} \eta^{r_d}(dz).
\nals
Assuming $f(z)$ to be a modular form the following SAGE functions (written using version 9.1 of the software \cite{sagemath}) help computing $[0]_{a/c} f$, the constant term of $f(z)$ at the cusp $a/c$.

{\scriptsize \begin{lstlisting}
def v_eta1(a,b,c,d): 
    if c%2==1:
        return kronecker_symbol(d,abs(c))
    if c%2==0:
        return kronecker_symbol(c,abs(d))

def v_eta2(a,b,c,d): 
    if c%2==1:
        return 1
    if c%2==0:
        return (-1)^(1/4*(sgn(c)-1)*(sgn(d)-1))

def v_eta3(a,b,c,d): 
    if c%2==1:
        return (1/24*((a+d)*c-b*d*(c^2-1)-3*c))
    if c%2==0:
        return (1/24*((a+d)*c-b*d*(c^2-1)+3*d-3-3*c*d))

def L_constr(m,d,c): #Proposition 2.1 of [12]
    x1=m*d/gcd(c,m)
    u1=-c/gcd(c,m)
    y1=0
    v1=0
    for i1 in range(-abs(x1*u1),abs(x1*u1)): 
        if gcd(i1,x1)==1 and (1+i1*u1)%x1==0 and ((1+i1*u1)/x1)%2==1:
            y1=i1
            v1=(1+i1*u1)/x1
            return [x1,y1,u1,v1]
            break

def A_find(d,c): #finds a suitable matrix
    for b in range(abs(d*c)): 
        if gcd(b,d)==1 and (1+b*c)%d==0:
            a=(1+b*c)/d
            return [a,b,c,d]
            break

def f_c_of_eta(m,d,c): #Constant term of the 
			#Dedekind eta function at -d/c
    A=A_find(d,c)
    L=L_constr(m,d,c)
    a=A[0]
    b=A[1]
    c=A[2]
    d=A[3]
    x=L[0]
    y=L[1]
    u=L[2]
    v=L[3]
    vv=-m*b*v-y*a
    OP1=v_eta1(x,y,u,v)
    OP2=v_eta2(x,y,u,v)
    OP3=v_eta3(x,y,u,v)
    OP4=(1/24/m*vv*gcd(c,m))
    OP5=(gcd(c,m)/m)^(1/2)
    return [OP1,OP2,OP3,OP4,OP5]

def first_coeff_of_eta_q(N,etaq,a,c): #Computes the constant term of 
                           #the eta quotient [r_1,...,r_d,...,r_N] 
                           #at cusp a/c
    d=-a
    divs=divisors(N)
    L=len(etaq)
    if sum(1/24/divs[i2]*(gcd(c,divs[i2]))^2*etaq[i2] for i2 in range(L))>0:
        return 0 #does the vanishing order analysis
    else:
        VV1=prod((f_c_of_eta(divs[i1],d,c)[0])^(etaq[i1]) for i1 in range(L))
        VV2=prod((f_c_of_eta(divs[i1],d,c)[1])^(etaq[i1]) for i1 in range(L))
        SS1=sum((f_c_of_eta(divs[i1],d,c)[2])*(etaq[i1]) for i1 in range(L))
        SS2=sum((f_c_of_eta(divs[i1],d,c)[3])*(etaq[i1]) for i1 in range(L))
        VV3=prod((f_c_of_eta(divs[i1],d,c)[4])^(etaq[i1]) for i1 in range(L))
        VV4=e^(2*pi*I*(SS1+SS2))
        kk=sum(r for r in etaq)/2
        return (-1)^kk*VV1*VV2*VV3*VV4 
\end{lstlisting}}

By Lemma \ref{reducer} it will be sufficient to compute the constant terms of the eta quotient $f_k(z)$ defined by \eqref{eqr_11} at a set of inequivalent cusps of $\Gamma_0(24)$, which is done below with the help of this code. The set
\bals
\{ 1/1, 1/2, 1/3, 1/4, 1/6, 1/8, 1/12, 1/24 \}
\nals
gives a complete set of inequivalent cusps of $\Gamma_0(24)$, see \cite[Corollary 6.3.23]{cohenbook}. Note that if $k$ is fixed then the code can handle the vanishing order analysis. For instance the output for the code
\begin{lstlisting}
k=3
etaq=[-2*k-1,2*k+1,2*k+1,0,0,2*k+1,2*k+1,-2*k-1]
print(first_coeff_of_eta_q(24,etaq,1,2))
\end{lstlisting}
will be $0$. However, here we are working with a general $k$, and therefore the order analysis has to be done manually. When $k \geq 1$, the vanishing orders of $f_k(z)$ is greater than $0$ at cusps $\{1/2,1/3,1/4,1/6,1/8,1/12\}$. Thus we have
\bals
[0]_{1/2}f_k=0,~[0]_{1/3}f_k=0,~[0]_{1/4}f_k=0 ,~[0]_{1/6}f_k= 0,~[0]_{1/8}f_k= 0 ,~[0]_{1/12}f_k=0.
\nals
To compute $[0]_{1/1} f_k$ and $[0]_{1/24} f_k$ we run the following code:
\begin{lstlisting}
k=var('k')
assume(k,'integer')
eta=[-2*k-1,2*k+1,2*k+1,0,0,2*k+1,2*k+1,-2*k-1]
print(first_coeff_of_eta_q(24,eta,1,1).simplify())
print(first_coeff_of_eta_q(24,eta,1,24).simplify())
\end{lstlisting}
The output will be:
\begin{lstlisting}
-I*6^(k + 1/2)*3^(-2*k - 1)*2^(-4*k - 2)*(-1)^k
1
\end{lstlisting}
Simplifying these we obtain 
\bals
& [0]_{1/1} f_k=- \frac{i^{2k+1} \sqrt{6}}{3^{k+1} 2^{3k+2}},~ [0]_{1/24} f_k=1.
\nals
Putting everything together, for all $k \geq 1$ we have
\bals
& [0]_{1/1} f_k =- \frac{i^{2k+1}\sqrt{6}}{3^{k+1} 2^{3k+2}},~ [0]_{1/2} f_k =0,~ [0]_{1/3}f_k =0,~ [0]_{1/4} f_k=0,\\
& [0]_{1/6}f_k =0,~ [0]_{1/8}f_k =0,~ [0]_{1/12}f_k =0, [0]_{1/24}f_k =1.
\nals


\begin{thebibliography}{0}
\bibitem{arenas}
{A. Arenas,}
\newblock{Quantitative aspects of the representations of integers by quadratic forms,}
\newblock{in: Number Theory, Alemania, ISBN 3-11-011791-6, pp. 7--14 (1989).}

\bibitem{rmfpaper}
{Z.S. Aygin,}
\newblock{Extensions of Ramanujan–Mordell formula with coefficients $1$ and $p$,}
\newblock{J. Math. Anal. Appl., { 465}, 690--702 (2018).}

\bibitem{sqfreepaper}
{Z.S. Aygin,}
\newblock{On Eisenstein series in $M_{2k}(\Gamma_0(N))$ and their applications,}
\newblock J. Number Theory, { 195}, 358--375 (2019).


\bibitem{kennethbook}
{B.C. Berndt, R.J. Evans and K.S. Williams,} 
\newblock {Gauss and Jacobi sums,} 
\newblock Wiley--Interscience, New York (1998).

\bibitem{cohenbook}
{H. Cohen and F. Str\"{o}mberg,} 
\newblock {Modular Forms A Classical Approach,} 
\newblock Graduate studies in mathematics, American Mathematical Society, Providence, Rhode Island (2017).

\bibitem{cooperbook}
{S. Cooper,}
\newblock {Ramanujan’s Theta Functions},
\newblock  Springer International Publishing AG, Switzerland (2017).

\bibitem{cooperrmf}
{S. Cooper, B. Kane and D. Ye,}
\newblock {Analogues of the Ramanujan--Mordell theorem},
\newblock  J. Math. Anal. Appl., { 446}, 568--579 (2017).

\bibitem{cremona}
{J.E. Cremona,}
\newblock {Algorithms for modular elliptic curves,}
\newblock Cambridge University Press, Cambridge (1992).


\bibitem{DiamondShurman}
{F. Diamond and J. Shurman,}
\newblock {A First Course in Modular Forms,}
\newblock Graduate Texts in Mathematics 228, Springer-Verlag (2004).

\bibitem{fine}
{N. Fine,} 
\newblock{Basic Hypergeometric Series and Applications,} 
\newblock{American Mathematical Society, Providence, RI (1988).}

\bibitem{iwaniec}
{H. Iwaniec,}
\newblock{Topics in Classical Automorphic Forms,}
\newblock{Grad. Stud. Math., vol. 17, American Mathematical Society, Providence, RI (1997).}

\bibitem{Kohler}
{G. K\"{o}hler,}
\newblock {Eta Products and Theta Series Identities,}
\newblock  Springer Monographs in Mathematics, Springer (2011).


\bibitem{miyake}
{T. Miyake,}
\newblock {Modular Forms,}
\newblock Springer-Verlag, Berlin (1989), translated from the Japanese by Yoshitaka Maeda.

\bibitem{mordell}
{L.J. Mordell,} 
\newblock{On the representations of numbers as a sum of $2r$ squares,}
\newblock Quart. J. Pure Appl. Math. { 48} (1917), 93--104.

\bibitem{opitz}
{S. Opitz.}
\newblock{Computation of Eisenstein series associated with discriminant forms, Doctoral thesis. Available at: \href{https://tuprints.ulb.tu-darmstadt.de/8261/1/20181203_Dissertation_Sebastian_Opitz.pdf}{https://tuprints.ulb.tu-darmstadt.de/8261/1/20181203\_Dissertation\_Sebastian\_Opitz.pdf}.}

\bibitem{ramanujan}
{S. Ramanujan,}
\newblock{On certain arithmetical functions}
\newblock{Trans. Cambridge Philos. Soc., { 22}, 159--184 (1916).}

\bibitem{sagemath}
{The Sage Developers,}
\newblock{{S}ageMath, the {S}age {M}athematics {S}oftware {S}ystem ({V}ersion 9.1),} 
\newblock {https://www.sagemath.org (2021).}

\bibitem{surveyRSP}
{R. Schulze-Pillot,} 
\newblock{Representation of Quadratic Forms by Integral Quadratic Forms. In: Alladi K., Bhargava M., Savitt D., Tiep P. (eds) Quadratic and Higher Degree Forms. }
\newblock{Developments in Mathematics, { 31}, Springer, New York, NY (2013).}

\bibitem{serrebook}
{J.-P. Serre,} 
\newblock {A course in arithmetic,}
\newblock Graduate texts in mathematics, Springer--Verlag, New York (1973).

\bibitem{siegel}
{C.L. Siegel,}
\newblock{Uber die analytische theorie der quadratischen formen,}
\newblock{Ann. of Math., { 36}, 527--606 (1935).}

\bibitem{wangpei}
{X. Wang, D. Pei,}
\newblock{Modular forms with integral and half-integral weights,}
\newblock{Science Press Beijing and Springer-Verlag, Berlin Heidelberg (2012).}

\bibitem{weis}
{J. Weisinger,}
\newblock{Some results on classical Eisenstein series and modular forms over function fields,} 
\newblock{Thesis (Ph.D.)–Harvard University, ProQuest LLC, Ann Arbor, MI, 1977.}

\bibitem{yang}
{T. Yang,}
\newblock {An explicit formula for local densities of quadratic forms,}
\newblock J. Number Theory {\bf 72}, 309--356 (1998).

\end{thebibliography}
\end{document}